\documentclass[10pt]{article}
\usepackage{amssymb}
\usepackage{amsmath}
\usepackage{amsthm}
\usepackage{epsfig}
\usepackage{mathrsfs}
\usepackage{xcolor}
\usepackage{graphicx}
\usepackage{color}
\usepackage{comment}

\newtheorem{theorem}{Theorem}[section]
\newtheorem{lemma}[theorem]{Lemma}
\newtheorem{proposition}[theorem]{Proposition}
\newtheorem{definition}[theorem]{Definition}

\newtheorem{remark}[theorem]{Remarks}
\newtheorem{rk&ex}[theorem]{Remarks \& Examples}
\newtheorem{corollary}[theorem]{Corollary}

\def\NN{{\mathbb N}}
\def\N{{\mathbb N}}
\def\ZZ{{\mathbb Z}}
\def\Z{{\mathbb Z}}
\def\RR{{\mathbb R}}
\def\R{{\mathbb R}}

\def\e{\varepsilon}

\begin{document}
	\title{Homogenization of oblique boundary value problems}
	\date{\vspace{-5ex}}
	
		\author{Sunhi Choi \thanks{Department of Mathematics, U. of Arizona, Tucson, AZ. }  and Inwon C. Kim
		\thanks{Department of Mathematics, UCLA, LA CA. Research supported by NSF DMS-1566578 }}
	\maketitle

	\begin{abstract}
		We consider a nonlinear Neumann problem, with periodic oscillation in the elliptic operator and on the boundary condition. Our focus is on problems posed in half-spaces, but with general normal directions that may not be parallel to the directions of periodicity. As the frequency of the oscillation grows, quantitative homogenization results are derived. When the homogenized operator is rotation-invariant, we prove the H\"{o}lder continuity of the homogenized boundary data. While we follow the outline of \cite{CK}, new challenges arise due to the presence of tangential derivatives on the boundary condition in our problem. In addition we improve and optimize the rate of convergence within our approach. Our result appear to be new even for the linear oblique problem.
		
	\end{abstract}

	\section{Introduction}

	For given $\e>0$, $\nu\in\mathcal{S}^{n-1}$ and  $\tau\in\R^n$, let $u_\e$ be a bounded solution of the following problem:
	$$
	\left\{\begin{array}{lll}
	F(D^2u_\e, \frac{x}{\e}) = 0 &\hbox{ in } \quad \Pi := \{x \in\R^n: -1< (x-\tau)\cdot\nu  <0\}  \\ \\
	u_\e =h(x)  &\hbox{ on } \quad  H_{-1}:= \{(x-\tau)\cdot\nu =-1\}\\ \\
	\partial_{\nu} u =   G(Du_\e, \frac{x}{\e}) &\hbox{ on } \quad  H_0:=\{(x-\tau)\cdot\nu = 0\}. 
	\end{array}\right.\leqno(P)_\e
	$$
	Here $F(M, y)$ and $G(p,y)$ are $\Z^n$-periodic in the $y$ variable. We also assume the boundary condition to be {\it oblique} and $F$ to be uniformly elliptic: see Section 1.1 for precise assumptions on $F$ and $G$. 
	
	\medskip
	
	The examples of boundary conditions we consider include the linear oblique problem
	\begin{equation}\label{oblique}
	\vec{\gamma}(\frac{x}{\e})\cdot Du +g(\frac{x}{\e})= 0,
	\end{equation}
	where the vector field $\vec{\gamma}$ satisfies  $c(\frac{x}{\e},x):=\vec{\gamma}(\frac{x}{\e})\cdot \nu >0$.  In this case one can write 
	$$G(p,y) = (c(y))^{-1}[\vec{\gamma}(y) \cdot p_T+ g(y)],
	$$ where $p_T$ is the tangential component of a vector field $p$ on $H_0$.  A nonlinear example is {\it capillarity-type} conditions, for which $G$ is given by
	\begin{equation}\label{nonlinear}
	G(p,y)= \theta(y) \sqrt{1+|p|^2},
	\end{equation}
	where $|\theta(x)|< 1$. This condition describes a prescribed contact angle between the graph $\Gamma:= \{(x,z): z=u(x)\}$ and the ``container boundary"  $H_0 \times \R$ with chemical inhomogeneities.

	\medskip
	
	We are interested in the behavior of $u_\e$ as $\e$ tends to zero.  Note that, as first pointed out by Bensoussan, Lions, and Papanicolaou \cite{BLP}, if $\nu$ is a multiple of a vector in $\Z^n$ (i.e. if $\nu$ is rational) then $\tau\cdot \nu$ must be zero for $u^\e$ to converge, since otherwise the Neumann boundary condition changes drastically as $\e$ changes, and thus $u_\e$ would not have a limit. When $\nu$ is irrational we expect $u_\e$ to average due to the ergodic property of its Neumann data. However in this case  $u^\e$ is no longer periodic, and thus interesting challenges arise in dealing with the inherent lack of compactness. Compared to \cite{CK} where linear Neumann problem was considered,   there is an additional challenge in our setting given by the presence of tangential derivatives on the boundary condition. We will discuss below some of the relevant literature on this issue.
	
	\medskip
	
	Let us state a convergence result on $(P)_\e$ to begin the discussion. let $\bar{F}$ be the homogenized operator of $F$ obtained by Evans \cite{evans}.
	
	\begin{theorem}\label{thm:main}
		Let $\nu$ be irratonal, or otherwise suppose $\tau=0$. Let us assume $(F1)-(F3)$ and $(G1)-(G3)$ (see Section 1.1). In addition suppose that $F(\cdot,x)$ is convex when $G(\cdot,x)$ is nonlinear. Then there exists $\mu(\nu, D_Tu)$, independent of $\tau$, such that $u_\e$  converges uniformly to the unique bounded solution $\bar{u}$ of 
		$$
		\left\{\begin{array}{lll}
		\bar{F}(D^2 \bar{u}) = 0 &\hbox{ in } &\Pi\;\ \\ \\
		\bar{u} = h(x)& \hbox{ on }& H_{-1}\\  \\
		\partial_\nu \bar{u}  = \mu(\nu, D_T\bar{u}) &\hbox{ on } &H_0.
		\end{array}\right. \leqno(\bar{P})
		$$
		Moreover $\mu(\nu,q)$  is Lipschitz continuous with respect to $q$. If $\bar{F}(M)$ is rotation-invariant, then $\mu$ is also H\"{o}lder continuous over irrational directions $\nu$ with exponent $\alpha = \frac{1}{5n}$. 
	\end{theorem}
	
	The proof of Theorem 1.1 will be given later in this section, based on our main result (Theorem 1.2), which establishes rates of convergence for (approximate) cell problem solutions. Our work extends the previous results in \cite{CKL} on linear Neumann problems where $G(p,y)=G(y)$. For general $G(p,y)$ additional challenges arise due to the presence of tangential derivatives on the boundary condition, which necessitates Lipschitz regularity estimates for the solutions.  As noted in \cite{FK}, the continuity property of $\mu(\nu, q)$ fails when $\bar{F}$ is not rotation-invariant, even when it is convex. When the continuity result holds for $\mu$ one can expect to proceed as in \cite{CK} to address general domains, but the analysis would require higher regularity estimates on the solutions, so we do not pursue this here.
	
	\medskip
	
	It is unknown whether the form of the boundary condition such as \eqref{oblique} or \eqref{nonlinear} is preserved in the limit $\e\to 0$. With the exception of linear problems, the interaction between the operator $F$ and the boundary condition remains to be better understood to yield further characterizations of the homogenized problems.

	\medskip
	
\noindent	{\bf Literature}
	
	\medskip
	
	Before proceeding further, let us briefly describe some of relevant literature.  In the classical paper of \cite{BLP}, the following problem was considered:
	\begin{equation}
	-\nabla\cdot(A(\frac{x}{\e}) \nabla u_\e)=0 \hbox{ in } \Omega, \quad\quad \nu\cdot(A(\frac{x}{\e})\nabla u_\e)(x)=g(\frac{x}{\e}) \hbox{ on } \partial\Omega.
	\end{equation}
	For this co-normal boundary value problem, explicit integral formulas have been derived for the limiting operator as well as for the limiting boundary data, under the assumption that $\partial \Omega$ does not contain any flat piece with a rational normal.
	
	\medskip
	
	For linear elliptic systems with either Dirichlet or Neumann problem with co-normal derivatives,  there has been a recent surge of development in quantitative homogenization by integral representation of solutions: we refer to \cite{AKP},\cite{GVM},\cite{SZ} and the references therein. 
	
	\medskip
	
	For nonlinear problems, or even for linear problems with non co-normal boundary data, most available homogenization results concern half-space type domains whose boundary goes through the origin and is normal to a rational direction. In \cite{T}, Tanaka considered some model problems in half-space whose boundary is parallel to the axes of the periodicity by purely probabilistic methods. In \cite{Ari} Arisawa studied specific problems in oscillatory domains near half spaces going through the origin. Generalizing the results of \cite{Ari} for nonlinear boundary conditions, Barles, Da Lio, Lions and Souganidis \cite{BDLS} studied the problem for operators with oscillating coefficients, in half-space type domains whose boundary is parallel to the axes of periodicity. We also refer to \cite{GS} which adopts an integro-differential approach to study linear scalar problems with the specific Neumann problem $G(p,y)=g(y)$. 
	
	\medskip
	
	For the linear Neumann problem $G(p,y)=g(y)$ in $(P)_\e$, corresponding results to Theorem~\ref{thm:main} -Theorem 1.2 have been recently shown in \cite{CKL}.   General domains has been considered in \cite{CK} based on the cell problem analysis in \cite{CKL}.  Corresponding results for the Dirichlet boundary data has been obtained in \cite{F}.  Lastly for general operator $F$, \cite{FK} discusses the generic nature of discontinuity for the homogenized boundary data, for either linear Neumann or Dirichlet problem.

	\medskip
	
\noindent	{\bf Cell problem}
	
	\medskip
	
	By the formal expansion $u_\e = \bar{u}(x) + \e v(x,\frac{x}{\e}) + O(\e^2)$, for a rational $\nu$, the cell problem for $v$ was derived in \cite{BDLS} for a rational $\nu$ and $\tau=0$. There they find a unique constant $\mu=\mu(\nu,q)$ for $q\in <\nu>^{\bot}$ such that the boundary value problem
	$$
	\left\{\begin{array}{lll}
	F(D^2 v, y) = 0 &\hbox{ in } & \{y\cdot\nu \geq 0\},\\
	\mu = G(Dv+p, y) &\hbox{ on } &H_0,
	\end{array}\right. \leqno (C)
	$$
	with $p= \mu\nu+ q$, has a bounded periodic solution $v$ in $\{y\cdot\nu \geq 0\}$.  The existence of bounded $v$ leads to the uniform convergence of $u_\e$ to $\bar{u}$  in the limit $\e\to 0$ with $p = D\bar{u}$ on $H_0$.
	
	\medskip
	
	For general $\nu$ and $\tau$, an approximate cell problem needs to be derived, since $v$ is no longer expected to be periodic and thus compactness is lost: see problem $(P)_{\e,\nu,\tau,q}$ below.  In the context of $(C)$, our result shows that for irrational $\nu$, there exists a unique constant $\mu = \mu(\nu, q)$ for $q\in <\nu>^{\bot}$ such that the problem 
	$$
	\left\{\begin{array}{lll}
	F(D^2 v, y+\tau) = 0 &\hbox{ in } & \{y\cdot\nu\geq 0\},\\
	\mu = G(Dv+p, y+\tau) &\hbox{ on } &H_0
	\end{array}\right. \leqno (\tilde{C})
	$$
	with any $\tau\in\R^n$ has a solution with sublinear growth at infinity. To show this, we use the ergodicity of Neumann data in a scale depending on $\nu$, and the stability of solutions under perturbation of boundary conditions. When the homogenized operator $\bar{F}$ is rotation-invariant, we show that $v$ is stable as the normal direction of the domain $\nu$ varies. A quantiative version of this stability property yields the mode of continuity for $\mu$ as $\nu$ varies.

	\medskip

\noindent	{\bf A discussion on assumptions on $F$ and $G$}
	
	\medskip
	
	Our assumptions on $F$ and $G$ are mainly to  obtain Lipschitz estimates for the solutions of $(\tilde{C})$. The Lipschitz estimates ensure that the solution of the cell problem has ergodic structure with respect to translations along the Neumann boundary (see Lemma~\ref{perturbation}), which happens when $\e$ changes in $(P)_\e$ and when $\tau$ is not the origin.  
	Already to guarantee the Lipschitz bound,  available literature restricts $F(M,x)$ to be convex with respect to $M$ when $G$ is a nonlinear function of $Du$. We refer to \cite{BDL} for a detailed description of available regularity theory on nonlinear Neumann boundary problems.  For the continuity properties of $\mu$ we further need $C^{1,\alpha}$ estimates for solutions of $(\tilde{C})$, however this does not further restrict the class of problems we can address. To deal with domains with general geometry, the approach taken in \cite{CK} or \cite{FK} uses fundamental solutions as barriers to bound the potential singularity generated at points with rational normals. For our problem, while we suspect our result to hold in general domains, we suspect that these singular solutions may cause new challenges in dealing with perturbative arguments, due to their singularity in tangential derivatives.

	\medskip

	\subsection{Assumptions and main results}

	Let $\mathbb{T}$ be the $1$-periodic torus in $\R^n$, and let $\mathcal{M}^n$ be the space of real  $n \times n$ symmetric matrices. Consider the functions  $F(M,y):\mathcal{M}^n \times \mathbb{T}\to \R$ and $G( p, y):\R^n \times\mathbb{T}$ satisfying the following properties:
	\begin{itemize}
		\item[(F1)] (Uniform Ellipticity)  There exist constants $0<\lambda<\Lambda$ such that
		$$\lambda {\rm Tr}(N) \leq F(M,y)-F(M+N,y) \leq \Lambda {\rm Tr}(N)$$  for all $y \in \mathbb{T}$ and $M, N \in \mathcal{M}^n$ with  $N\geq 0$.  
		\item[(F2)] (1-Homogeneity)  $F(t M,y) = tF(M,y)$ for all $y \in \mathbb{T}$,   $t>0$ and  $M\in\mathcal{M}^{n}$. 
		\item[(F3)] (Lipschitz Continuity) There exists $C>0$ such that for all $y_1,y_2 \in \mathbb{T}$ and $M, N \in \mathcal{M}^n$, 
		$$|F(M,y_1) -F(N,y_2)| \leq C (|y_1-y_2|(1+\|M\| +\|N\|) + \|M-N\|) .$$
		\item[(G1)] (At most linear Growth) $|G(p,x)| \leq \mu_0(1+|p|)$.
		\item[(G2)] (Lipschitz continuity) $(1+|p|)|G_p|,  |G_y| \leq m(1+|p|)$ for some $m>0$.
		\item[(G3)] (Oblicity) $|G_p \cdot \nu| \leq c< 1$.
	\end{itemize}
	
	\medskip

	A typical example of an operator $F$ satisfying (F1)-(F3) is the linear elliptic operator
	\begin{equation}\label{div_elliptic}
	F(D^2u, x) = -\Sigma_{i,j} a_{ij}(x) \partial_{x_ix_j} u ,
	\end{equation}
	where $a_{ij}:\R^n\to \R$ is  periodic and Lipschitz continuous. A nonlinear example is the Bellman-Isaacs operator arising from stochastic optimal control and differential games
	\begin{equation}\label{BI}
	F(D^2u, x) = \inf_{\beta\in B}\sup_{\alpha\in A} \{\mathcal{L}^{\alpha, \beta} u\},
	\end{equation}
	where $\mathcal{L}^{\alpha,\beta}$ is a family of uniformly elliptic operators of the form \eqref{div_elliptic}. In fact, all operators satisfying (F1)-(F3) can be written as \eqref{BI}. As for $G$, the ones given in \eqref{oblique} and \eqref{nonlinear} with Lipschitz coefficients $c^{-1}\vec{\gamma}, c^{-1}g$ and $\theta$ satisfy (G1)-(G3).

	\medskip

	For $\tau\in\R^n$ and $\nu\in \mathcal{S}^{n-1}$, let us define the strip domain 
	$$\Pi (\tau, \nu) := \{x \in \R^n : -1 \leq (x-\tau) \cdot \nu  \leq 0\}, \quad H_s := \{(x-\tau) \cdot \nu  =s\}.
	$$
	For a given $q\in <\nu>^{\bot}$, let $u_\e$ solve the following approximate cell problem
	$$
	\left\{\begin{array}{lll}
	F(D^2 u_\e,\frac{x}{\e})=0 &\hbox{ in }& \Pi( \tau, \nu)\\ \\
	\partial_\nu u_\e = G(Du_\e,\frac{x}{\e})  &\hbox{ on }& H_{0}\\ \\
	u_\e(x) =  q \cdot x &\hbox{ on  }& H_{-1}
	\end{array}\right.\leqno (P)_{\e, \nu, \tau, q}.
	$$
	
	Now we are ready to state the main result.

	\begin{theorem} \label{main-cell-1} 
		Let $u_\e$ solve $(P)_{\e,\nu, \tau,q}$. Suppose that either $\nu$ is irrational or $\tau=0$. Then the following holds:
		
		\begin{itemize}
			\item[(a)] There exists $\mu=\mu(\nu,q)$ such that $u_\e$ converges uniformly to the linear profile
			$$u(x) := \mu((x-\tau)\cdot \nu + 1) + q\cdot x.$$
			Here, $\mu(\nu,q)$ is independent of $\tau$ and Lipschitz continuous with respect to $q$. Moreover we have
			\begin{equation}\label{rate}
			|u_\e- u| \leq C\Lambda(\e, \nu) \quad \hbox{ in } \Pi(\tau,\nu),
			\end{equation}
			where $\Lambda(\e, \nu)$ (as given in \eqref{TheFunction}) 
				 is an increasing function of $\e$ such that $\displaystyle{\lim_{\e\to 0}\Lambda(\e, \nu) =} 0$. 
			 
			\item[(b)] When $\bar{F}$ is rotation-invariant, there exists a continuous extension $\bar{\mu}(\nu,q): \mathcal{S}^{n-1}\times \R^n \to \R$  of $\mu(\nu,q)$ over irrational directions $\nu \in \mathcal{S}^{n-1} - \R\Z^n$.  Moreover
			 $\bar{\mu}$  is Lipschitz in $q$ and $C^{\alpha}$ in $\nu$, with $\alpha = \displaystyle{\frac{1}{5n}}$.

		\end{itemize}
	\end{theorem}
	
	The proof is given in Theorem~\ref{main-cell-2}, Theorem~\ref{continuity} and Theorem~\ref{lipinq}.
	
	\medskip

	\medskip
	
	\noindent	{\bf A discussion on the rate of convergence  $\Lambda(\e,\nu)$}
	
	\medskip

Here we briefly describe the geometric process used in section 4 to obtain an upper bound for the rate function $\Lambda$ in \eqref{rate}.  Given $\delta>0$, we are interested in finding $\e_0 =\e_0(\nu, \delta)$ such that 
	$|u_\e -u| \leq C \delta \,\,\hbox{ for }\,\, \e \leq \e_0.$
	
	\medskip
	
If $\nu$ is rational and $\tau=0$, $F$ and $G$ are periodic along $\nu$-direction with period $T_\nu$. Hence we expect that $\e_0$ needs to be smaller than $1/T_{\nu}$ for a fixed $\delta$. In fact Theorem~\ref{main-cell-2} (d) yields that
		$$\Lambda(\e, \nu) \leq  \delta \,\,\hbox{ for } \,\,\e \leq \e_0 = \delta^2/T_\nu $$
		and thus yields a uniform bound 
		\begin{equation}\label{unib}
		\Lambda(\e, \nu)   \leq C(\nu)\e^{1/2}.
		\end{equation}

If $\nu$ is irrational, for each $\delta$ we choose a reference rational direction $P$ as follows: choose a point $P=P(\nu, \delta) \in \Z^n$ such that 	
\begin{equation} \label{tdelta}
	|T\nu -P|\leq \delta \,\, \hbox{ for some } \,\, T=T(\nu,\delta)>0.
\end{equation}	
	 Then $F$ and $G$ are periodic along $P$-direction with period $T+O(\delta)$.  If we let $\theta = \theta(\nu, \delta)$ be the angle between $\nu$ and $P$, then (\ref{tdelta}) can be written as $\theta < \delta/T$.
	 If $R< 1/\theta$, then due to the proximity of $\nu$ to $P$ direction, $G(p,\cdot)$ takes only limited values of $G$ on $H_0\cap B_R(\tau)$, even though $\nu$ is irrational. In other words $G(p,\cdot)$ exhibits ergodicity on $H_0$ only in a neighborhood of size $R> 1/\theta$. For this reason  $u_\e$ homogenizes only when $\e \leq O(\theta)$. Indeed Theorem~\ref{main-cell-2} (c) yields that
		$$\Lambda(\e, \nu) \leq  \delta \,\, \hbox{ for } \,\, \e \leq \e_0 = \delta^2 \theta.$$
	Since $\theta$ depends on not only $\nu$ but $\delta$, we are not able to separate the dependence of the rate function on $\e$ and $\nu$, without further estimate of $\theta$ or $T$ as $\delta$ varies. Such estimate would require better understanding of  the {\it discrepancy } function discussed in the Appendix.

	\medskip

\noindent {\bf Proof of Theorem~\ref{thm:main}}
	
\medskip
	
	Once Theorem~\ref{main-cell-1} (a) is obtained, one can derive our main theorem by the {\it perturbed test function} arguments introduced by Evans \cite{E}.

	\medskip
	
	Let $u_\e$ solve $(P)_\e$, and define $u^*$ and $u_*$ as
	$$
	u^*= \mathord{\limsup}^* u_\e:= \lim_{r\to 0} \sup_{(y,\e)\in S_r^x} u_\e(y); \quad\,\,\, u_*= \mathord{\liminf}_* u_\e:= \lim_{r\to 0} \inf_{(y,\e)\in S_r^x} u_\e(y),
	$$
	where $S_r^x = \{(y,\e): y\in\Pi, |x-y|<r, 0<\e<r\}$.
	First, observe that, by using a barrier of the form 
	$$
	\varphi_M(x):= M((x-\tau)\cdot\nu+1)  + f(x),
	$$
	where $f$ is a $C^2$-approximation of $h$ that is larger than $h$, one can conclude that $u_\e \leq \varphi_M$ in $\Pi$ for any large $M$, and thus $u^* \leq h$ on $H_{-1}$. Similar arguments yield that $u_*\geq h$ on $H_{-1}$.
	
	\medskip
	
	We claim that $u^*$ and $u_*$ are respectively a viscosity subsolution and  supersolution of $(P)$. If the claim is true, then Corollary~\ref{thm:cp2}  applies to yield that $u^*\leq u_*$. Since the opposite inequality is true from the definition, we conclude that $u^*=u_*$, which means that $u_\e$ uniformly converges in $\bar{\Omega}$.
	
	\medskip
	
	Below we will only show that $u^*$ is a subsolution of $(P)$, since the proof for $u_*$ can be shown by parallel arguments. To this end,
	suppose that $u^* - \phi$ has a local max in $B_r(y_0) \cap \bar{\Pi}$ with a smooth test function $\phi$.  If $y_0$ is in the interior of $\Pi$, then $\bar{F}( D^2\phi)(y_0) \leq 0$ due to standard interior homogenization (see for instance \cite{E}). Hence it remains to show that if $y_0$ is on the Neumann boundary then $\phi$ satisfies 
	\begin{equation}\label{iinequality}
	\partial_{\nu}\phi \leq \mu(\nu, q:=D^T\phi) \hbox{ at } x=y_0.
	\end{equation}

	First suppose that $\nu$ is rational and $y_0\cdot \nu = 0$. We may assume for simplicity that $u(y_0)=\phi(y_0)=0$ and define $P(x) := D\phi(y_0)\cdot (x-y_0)$. Since $\Pi \subset \{x: x\cdot \nu < 0\}$, for any $\delta>0$ we may choose $r$ sufficiently small that  $l_{\delta}(x):= P(x) -\delta(x\cdot\nu)$ is strictly larger than $u^*$ on $B_r(0)\cap\Pi$. Then for sufficiently small choice of $\e$ we have
	\begin{equation}
	\label{iiin}
	l_{\delta} > u_\e \hbox{ on }  B_r(0) \cap H_{-r\delta}, \quad \hbox{ where } H_{-r\delta} = \{x \cdot \nu = -r \delta\}.
	\end{equation}
	
	Let  $\bar{\e}:= (r\delta)^{-1}\e$ and consider the re-scaled function $v_\e(x):= (r\delta)^{-1} u_\e( r\delta x) -l_{\delta}(x)$. Then $v_\e $ is a subsolution of $(P)_{\bar{\e},\nu, 0, q}$, in the local domain $\Pi \cap B_{\delta^{-1}}(0)$. Note that the corresponding Neumann boundary for $v_\e$ remains to be $H_0$ since $y_0\cdot \nu = 0$: in general it will be $\{(x-\tau) \cdot \nu  =0\}$  with 
	\begin{equation}\label{translation}
	\tau =   (\bar{\e})^{-1} y_0,
	\end{equation}
	and thus the choice of $\tau$ must change as we vary $\bar{\e}$. We will compare $v_\e$ with $w_{\bar{\e}}$, the unique bounded solution of $(P)_{\bar{\e},\nu, 0, q}$ in $\Pi$ obtained in Lemma 3.3.
	Due to the localization lemma (Lemma~\ref{lem:side}) we have
	\begin{equation} \label{88}
	v_\e  \leq w_{\bar{\e}} +M\delta.
	\end{equation}

	Due to Theorem~\ref{main-cell-1} we have
	$$
	w_{\bar{\e}} \leq  \mu(\nu,q) (x\cdot \nu +1) + q\cdot x +  \Lambda(\bar{\e},\nu) \hbox{ in } \Pi.
	$$
	Since $\Lambda(\e,\nu)\to 0$ as $\e\to 0$, \eqref{iiin} and (\ref{88}) yield that
	\begin{equation}\label{med}
	\limsup_{\e\to 0}(r\delta)^{-1}u_\e(r\delta x) = \limsup_{\e\to 0} v_\e (x) + l_{\delta}(-\nu)  \leq   \mu(\nu,q) (x\cdot \nu +1) + q\cdot x +l_{\delta}(-\nu) + M \delta \hbox{ in } \Pi.
	\end{equation}

	\medskip
	
	Now suppose that \eqref{iinequality} is false, then there exists $\delta>0$ such that 
	\begin{equation}\label{contradiction}
	\partial_{\nu}\phi(0) = \delta - l_{\delta}(-\nu)  > \mu(\nu, q) + (M+1)\delta.
	\end{equation} 
	This means that the right side of \eqref{med} is strictly negative at $x=0$, which contradicts the assumption that $u^*(0)=0$.

	\medskip
	
	Next suppose that $\nu$ is irrational, we need to choose $\tau$ depending on $\bar{\e}$ so that \eqref{translation} holds. Then  we argue as above with a solution of $(P)_{\bar{\e}, \nu, \tau, q}$ in $\Pi$. Here we must use the fact that $\nu$ is irrational and thus Theorem 1.2 ensures the uniform convergence of $w_{\bar{\e}}$ to the linear profile is regardless of the choice of $\tau$.

	\hfill$\Box$
	
	\section{Preliminaries}

	We adopt the following definition of viscosity solutions, which is equivalent to the one given in \cite{CIL}. Let $\Omega$ be domain in $\R^n$ with $\partial \Omega$ as a disjoint union of  $\Gamma_0$ and $\Gamma_1$. Let $F$ satisfy (F1) - (F3) in the previous section, and  let $G$ satisfy $(G3)$ with $G(p,x)$ being uniformly continuous in $p$ independent of the choice of $x$.   For $f \in C(\Gamma_0)$ consider the following problem
	$$
	\left\{\begin{array}{lll}
	F(D^2 u,x)=0 &\hbox{ in }& \Omega\\ \\
	u=f(x) &\hbox{ on }& \Gamma_0\\ \\
	
	\frac{\partial}{\partial\nu} u = G(Du,x)  &\hbox{ on }& \Gamma_1
\end{array}\right.\leqno (P)
$$
where $\nu=\nu(x)$ is the outward unit normal at $x \in \Gamma_1$. Here we replace $(G3)$ with 

\vspace{10pt}

(G3)' (Oblicity) $|G_p\cdot\nu| \leq c<1$ on $\partial\Omega$, where $\nu = \nu_x$ is the outward normal at $x\in \partial\Omega$.

\begin{definition}
	\begin{itemize}
		\item[(a)]  An upper semi-continuous function $u:\bar{\Omega}\to \RR$ is a {\rm viscosity subsolution} of $(P)$ if $u$ cannot cross from below any $C^2$ function $\phi$ which satisfies 
		$$
		\left\{\begin{array}{l}
		F(D^2\phi, x)>0 \hbox{ in } \Omega, \quad \phi > f \hbox{ on } \Gamma_0,\\
		\nu\cdot D\phi >G(D \phi,x) \hbox{ if  } \tau\in\Gamma_1.
		\end{array}\right.
		$$
		\item[(b)] A lower semi-continuous function $u:\bar{\Omega}\to \RR$ is a {\rm viscosity supersolution} of $(P)$ if
		if $u$ cannot cross from above any $C^2$ function $\varphi$ which satisfies
		$$
		\left\{\begin{array}{l}
		F(D^2\phi, x ) <0 \hbox { in } \Omega,  \quad \phi < f \hbox{ on } \Gamma_0,\\
		\nu\cdot D\phi < G(D \phi,x) \hbox{ on } \Gamma_1.
		\end{array}\right.
		$$
		\item[(c)] $u$ is a {\rm viscosity solution} of $(P)$ if its upper semi-continuous envelope $u^*$ is a viscosity subsolution and its lower semi-continuous envelope $u_*$ is a viscosity supersolution of $(P)$. 
	\end{itemize}
\end{definition}

Existence and uniqueness of viscosity solutions of $(P)$ are based
on the comparison principle we state below. We refer to
\cite{CIL} and \cite{il} for details on the proof of the following theorem as well as the
well-posedness of the problem $(P)$.

\begin{theorem}\label{thm:comp}
	Let $G$ and $F$ satisfy the conditions (G1) and (G3) and (F1) - (F3) in the previous section, with $G$ being uniformly continuous in $p$ independent of the choice of $x$.
	Let $u$ and $v$ be respectively bounded viscosity subsolution and supersolution of $(P)$ in a bounded domain $\Omega$. Then $u \leq v$ in $\Omega$.
\end{theorem}

For  a symmetric $n\times n$ matrix $M$, we decompose $M=M_+ -M_-$ with $M_{\pm} \geq 0$ and $M_+M_- =0$. We define the Pucci operators as 
$$
\mathcal{P}^+(M) =- \Lambda tr(M_+) + \lambda tr(M_-)
$$
and
$$
\mathcal{P}^-(M)=-\lambda tr(M_+) +\Lambda tr(M_-)
$$
where $0 < \lambda < \Lambda$.
Later in the paper we will utilize the fact that the difference of two solutions of $F(D^2u, x)=0$ is both a subsolution of $\mathcal{P}^+ = 0$ and a supersolution of $\mathcal{P}^- =0$.
(see \cite{CC}).

\medskip

Next we state some regularity results that will be used throughout this paper.

\begin{theorem}\label{lemma-reg0}[Chapter 8, \cite{CC}, modified for our setting]
	Let $u$ be a viscosity solution of $F(D^2u,x)=0$ in a domain $\Omega$.
	Then for any compact subset $\Omega'$ of
	$\Omega$, we have
	$$
	\|Du\|_{L^\infty(\Omega')} \leq Cd^{-1}\|u\|_{L^\infty(\Omega)},
	$$
	where $d=d(\Omega',
	\partial\Omega)$ and $C>0$ depends on $n$, $\lambda$ and $\Lambda$. 
\end{theorem}

As mentioned in the introduction, regularity results for nonlinear Neumann problem is rather limited. $C^{0,\alpha}$ estimates has been obtained by Barles and Da Lio in general framework \cite{BDL}. While a priori results for the gradient bounds are available for general $F$ and $G$ in \cite{LT}, their results are based on linearization and thus require existence of classical solutions. For $G(p,x)$ that is linear in $p$,  regularity estimates on $Du$ are recently obtained by Li and Zhang \cite{LZ}.

\begin{theorem}\label{thm:reg2} \cite{LZ}, \cite{LT}
	Let $u$ be a viscosity solution of $(P)$ with $|u| \leq M$. 
	$$
	B_r^+:= \{|x|< r\} \cap\{x\cdot e_n \geq 0\}\hbox{ and }\quad\Gamma
	:= \{x\cdot e_n = 0\}\cap B_1.
	$$
	Let $u$ be a viscosity solution of
	$$
	\left\{\begin{array}{lll}
	F(D^2u,x)=0 &\hbox{ in }& B_1^+ \\
	\nu\cdot Du = G(Du,x) &\hbox{ on }&\Gamma.
	\end{array}\right.
	$$
For $F$ and $G$ satisfying $(F1)-(F3)$ and $(G1)-(G3)$, suppose that either (A) $F(M,x)$ is convex with respect to $M$, or (B) $G(p,x)$ is linear with respect to $p$. Then for any $0<\alpha<1$ we have 
	\begin{equation}\label{reg2_1}
	\|u\|_{C^{0,\alpha}(B_{1/2}^+)} , \|Du\|_{C^{0,\alpha}(B_{1/2}^+)} \leq C,
	\end{equation}
	where $C$ depends on $\alpha$ and $M$ as well as the constants given in $(F1)-(F3)$ and $(G1)-(G3)$.
\end{theorem}
Our proof extends in general to the cases where the estimate \eqref{reg2_1} holds for some $\alpha>0$.

\medskip

Lastly we mention interior homogenization result from \cite{CK}, which is a modified version of homogenization results such as in \cite{evans}.

\begin{theorem}\label{ext} (Theorem 2.14, \cite{CK})
	Let $K$ be a positive constant and let $f:\R^n\to\R$ be bounded and H\"{o}lder continuous. Given $\nu\in\mathcal{S}^{n-1}$, let $u_N:\{-K\leq x\cdot\nu\leq 0\}\to \R$ be the unique bounded viscosity solution of  
	$$
	\left\{\begin{array}{lll}
	F(D^2 u_N, Nx)=0 \hbox{ in } \{-K\leq x\cdot\nu \leq 0\};\\ \\
	\nu \cdot D u_N = f(x) \hbox{ on } \{x\cdot\nu=0\}, \qquad u = 1 \hbox{ on }  \{x\cdot\nu=-K\} .
	\end{array}\right.\leqno(P_N)
	$$
	Then for any $\delta>0$, there exists $N_0$  depending only on $K$, the bound of $u_N$ and  the H\"{o}lder exponent of $f$, such that
	\begin{equation}\label{conv}
	|u_N - \bar{u}| \leq \delta \hbox{ in } \{|x|\leq K\} \quad\hbox{ for } N \geq N_0,
	\end{equation}
	where $\bar{u}$ is the unique bounded viscosity solution of 
	$$
	\left\{\begin{array}{lll}
	\bar{F}(D^2 \bar{u})=0 \hbox{ in } \{-K\leq x\cdot\nu \leq 0\};\\ \\
	\nu \cdot D \bar{u} = f(x) \hbox{ on } \{x\cdot\nu=0\}, \qquad u=1 \hbox{ on }  \{x\cdot\nu = -K\}.
	\end{array}\right.
	$$
\end{theorem}

\section{Localization Lemmas } 
In this section we prove several lemmas on perturbing and localizing the solutions, which will be used frequently throughout the paper.
Below we prove a localization lemma, and as a corollary, we prove existence and uniqueness of  solution  $u_\e$ of $(P)_{\e, \nu, \tau, q }$ with $\Pi = \Pi(\nu,\tau)$ for   $\tau \in \R^n$ and   $\nu \in \mathcal{S}^{n-1}$. Denote $B_R(\tau):=\{|x-\tau|\leq R\}$ and recall $ H_s := \{(x-\tau) \cdot \nu  =s\}$.

\medskip
 
 First we state a basic lemma which will be frequently used. The proof is a direct consequence of the oblicity assumption $(G3)$.
 
 \begin{lemma}\label{bound}
There exists $M = M(|q|,c)$ such that $q \cdot x \pm Mx\cdot \nu$ are respectively super and subsolution of $(P)_{\e,\nu,\tau, q}$. 
\end{lemma}

\begin{lemma}\label{lem:side}
Let $f \in C(\R^n)$ be bounded. Suppose $w_1$ and
$w_2$ solve, in the viscosity sense,
\begin{itemize}
\item[(a)]  $ F(D^2w_i,\frac{x}{\e})= 0$ \quad  in \quad $\Sigma_R:=\Pi \cap B_R(0)$ \quad  for $i=1,2$
\item[(b)]  $\nu \cdot  Dw_i = G(D w_i, \frac{x}{\e})$ \quad   on \quad $H_{0}$ \quad for $i=1,2$
\item[(c)] $w_1=w_2\quad\hbox{ on }\quad H_{-1}$
\item[(d)] $0 \leq w_2-w_1  \leq M$\quad  on  \quad $\Pi\cap \partial B_R(0)$.
\end{itemize}
Let $L:=  \|G_p\|_\infty$ and $0<c<1$ is the constant given in $(G3)$.
Then there exists a constant $C(\frac{\Lambda}{\lambda}, c, L)>0$ such that 
$$w_1 \leq w_2 \leq w_1 + \frac{CM}{(1-c)R} \quad\hbox{ in }\quad \Pi\cap B_1(0).
$$
\end{lemma}

\begin{proof}
Without loss of generality, let us set $\nu =e_n$ and $\tau=0$. The
first inequality, $w_1\leq w_2$, directly follows from Theorem~\ref{thm:comp}.
 To show the second inequality, let $$w := w_1+ M(h_1+h_2)+ C_1h_3,$$ where
$$
h_1 (x) = \frac{|x|^2}{R^2} ,\quad  h_2 (x) =
\frac{C}{R^2} (1-(x_n)^2) \hbox{ with } C = \dfrac{n\Lambda}{\lambda}, \quad
h_3(x) = \frac{1+x_n}{R}, 
$$
 and   $C_1>0 $ is a large constant depending on  $n$, $\Lambda$, $\lambda$, $L$ and $c$, which will be chosen below in the proof.   
 
  Note that in  $\Sigma_R$,
$$
\begin{array}{lll}
F(D^2 w, \frac{x}{\e}) &=& F(D^2 w_1 + M(D^2 h_1 + D^2h_2), \frac{x}{\e})  \\ \\
&\geq& F(D^2w_1, \frac{x}{\e}) - \mathcal{P}^+(M(D^2 h_1 + D^2h_2)) \\ \\ &=&
F(D^2 w_1, \frac{x}{\e}) =0.
\end{array}
$$
Also $w_2 =w_1 \leq w$ on $H_{-1}$ and $w_2 \leq w_1 + M \leq w$  on $ \partial B_R(0) \cap \Pi$.

\medskip

 Hence to show that   $w_2\leq w$, it is enough to show that $\partial_{x_n} w \geq G(Dw, \frac{x}{\e})$ on $H_0$. We will verify that this is true when $C_1$ is sufficiently large.
Observe that in $\Sigma_R$  
\begin{equation} \label{DDD}
|D(h_1+h_2)|\leq 
\frac{C_0}{R} \,\,\hbox { for }\,\, C_0=C_0(n, \Lambda, \lambda).
\end{equation}
Hence on $H_0 \cap \Sigma_R$ we have
$$
\begin{array}{lll}
 \partial_{x_n} w &\geq&
\partial_{x_n} w_1 + \dfrac{C_1}{R} -\dfrac{C_0}{R} 
\\ \\ &=& G(Dw_1, \frac{x}{\e}) + \dfrac{(C_1-C_0)}{R}\\ \\
&\geq& G(Dw, \frac{x}{\e}) - \dfrac{cC_1}{R} + \dfrac{C_0 L}{R} + \dfrac{(C_1-C_0)}{R}
\end{array}
$$
where the last inequality follows from the Lipschitz property of $G$ with (\ref{DDD}), if $C_1 =C_1( n, \Lambda, \lambda , c  )$ is chosen  sufficiently large.  It follows from Theorem~\ref{thm:comp} that  $w_2 \leq w$ in $\Sigma_R$, and we obtain the lemma.
\end{proof}

As a corollary of Lemma~\ref{lem:side}, we prove existence and uniqueness of solutions in strip regions. 
\begin{lemma}
There exists a unique solution $u_\e$ of $(P)_{\e,\nu, \tau, q}$ such that $$\|u_\e- q \cdot x\| \leq M.$$  
\end{lemma}

\begin{proof}
1. Let $\Sigma_R$ be as given in Lemma~\ref{lem:side}, and consider
the viscosity solution $w_R(x)$ of $(P)_{\e,\nu, \tau,q}$ in $\Sigma_R$ 
 with the lateral boundary data $q \cdot x$  on $\partial B_R(\tau)\cap
 \Pi$.
 The existence and uniqueness of the viscosity solution $w_R$ is shown,
for example, in \cite{CIL} and \cite{il}.

From Lemma 3.1, $q \cdot x \pm M (x-\tau+\nu)\cdot \nu$ is  respectively a sub and supsersolution of $(P)_{\e,\nu, \tau, q}$, and thus by comparison principle we obtain that
$$
|w_R(x) - q \cdot x | \leq M\hbox{ for } x\in \Sigma_R.
$$
  Due to Theorem 2.5 and the Arzela-Ascoli Theorem,
$w_R$ locally uniformly converges to a continuous function
$u_\e(x)$.  From the stability property of viscosity solutions it
follows that $u_\e(x)$ is a viscosity solution of $(P)_{\e,\nu,\tau, q}$.

\vspace{5pt}

2. To show uniqueness, suppose $u_1$ and $u_2$ are both viscosity
solutions of $(P)_{\e,\nu, \tau, q}$ with $|u_1- q \cdot x|, |u_2- q \cdot x| \leq M$. Then
Lemma~\ref{lem:side} yields that, for any point $s\in H_{0}$
$$
|u_1-u_2| \leq O(1/R) \quad\hbox{ in }\quad B_1(s)\cap \Pi.
$$
Hence $u_1=u_2$.
\end{proof}

The following is immediate from Theorem~\ref{thm:comp} and the
construction of $u_\e$ in the above lemma.

\begin{corollary}\label{thm:cp2}
Suppose $u,v$ are bounded and continuous functions in
$\overline{\Pi} = \overline{\Pi (\tau,\nu)}$. In addition suppose they satisfy, for $F$ satisfying $(F1)-(F3)$ and $G$ satisfying $(G_1)-(G_2)$, 
\begin{itemize}
\item[(a)] $F(D^2 u, \frac{x}{\e}) \leq 0 \leq F(D^2v, \frac{x}{\e}) \quad\hbox{  in } \quad\Pi$;
\item[(b)] $u \leq v \quad\hbox{ on }\quad H_{-1}$;
\item[(c)]  $\nu \cdot Du \leq G(Du, x/\e)
; \quad  \nu \cdot D v \geq
G(Dv, x/\e)$ on $H_{0}.$
\end{itemize}

 Then $u \leq v$ in $\Pi$.
\end{corollary}

\begin{lemma}\label{perturbation}
There exists $C>0$ such that the following holds: let $u_1$ and $u_2$ be solutions of 
$$
\left\{\begin{array}{lll}
\mathcal{P}^+(D^2u_1)\leq 0, \,\,\,\, \mathcal{P}^-(D^2u_2) \geq 0 &\hbox{ in }& \Pi \cap B_R(0)\\ \\
\partial_\nu u_i = G_i(Du_i, x) &\hbox{ on }&H_0 \cap B_R(0) \\ \\
u_i = q \cdot x &\hbox{ on } & H_{-1} \cap B_R(0) 
\end{array}\right.
$$
where $\Pi = \Pi(\nu,0)$.  
Furthermore suppose that $G_i$ satisfies the assumption in Theorem~\ref{thm:reg2} and $G_1$ and $G_2$ satisfy 
\begin{equation}\label{Lipschitz}
|G_1(p,x) -G_2(p,x)| \leq \delta(1+|p|) \hbox{  and }|u_1-u_2|\leq M.
\end{equation}
 Let $L$ denote the Lipschitz bound for $u_i$ and $G_i's$. Then there exists $C = C(\Lambda, \lambda, n)$ such that  
$$
|u_1-u_2| \leq  \delta(L+1)+CM/R\hbox{ in }  \Pi \cap B_1(0).
$$ 
\end{lemma}

\begin{proof}

By our assumption, $v := (u_1-u_2)/M$ satisfies $|v|\leq 1$ in $B_R(0)$ with 
$$
\left\{\begin{array}{lll}
\mathcal{P}^+(D^2v) \leq 0 &\hbox{ in }& \Pi \cap B_R(0)\\
v=0 &\hbox{ on } & H_{-1} \cap B_R(0)
\end{array}\right.
$$

After a change of coordinates we may assume  $\nu=e_n$ so that $\Pi = \{x: -1\leq x_n \leq 0\}$, and we denote $x = (x', x_n)$. Define 
$$
w(x):= (c_0/M+c_1/R) (x_n+1)+  2(|x'|^2 - \dfrac{\Lambda n}{\lambda} (|x_n|^2-1)) /R^2.
$$ 
where $c_0$ and $c_1>8$ will be chosen later.
Then  $w$ is a supersolution of above problem with the Neumann boundary condition 
$$
\partial_n w = (c_0/M + c_1/R)  \geq (c_0/M+ 4|x'|/R^2) = (c_0/M+|D_Tw|) \quad \hbox{ on } \{x_n=0 \} \cap B_R(0).
 $$
  Now suppose $v-w$ has positive maximum in $\Pi \cap \overline{B_R}(0)$. Then the maximum would need to be achieved at a point $\tau \in H_0 \cap B_R(0)$. At this point we should have $\partial_n (v-w)  \geq 0$ and $D_T v = D_T w$.  Therefore 
 \begin{equation}\label{inequality}
\partial_n v \geq \partial_n w \geq ( c_0/M +|D_Tw|) =(c_0/M + |D_Tv|) \hbox{ at } x=\tau, 
 \end{equation}
 On the other hand
 $$
 G_1(Du_1,x) - G_2(Du_2,x) = DG_1 (p^*,x) \cdot D (Mv) + G_1(Du_2,x) - G_2(Du_2,x),                                      
 $$
 and since $|DG_1(p^*,x) \cdot e_n| \leq c $ we have, from \eqref{Lipschitz} and the Lipschitz bound for $u_i$ given in Theorem~\ref{thm:reg2},
  $$
 (1-c)\partial_n v \leq L|D_T v| +\frac{1}{M}|G_1(Du_2,x) - G_2(Du_2,x)| \leq 
  L|D_Tv| +  \frac{\delta}{M}(L+1)  \hbox{ at } x=\tau.
 $$ 
  
Then using the fact that $|D_Tw| = 4|x'|/R^2\leq 4/R$ in $B_R(0)$ it follows that 
\begin{equation}
(1 - c) |\partial_n v| \leq \frac{4L}{R} +  \frac{\delta(L+1)}{M}.
\end{equation}
Hence  from \eqref{inequality} we get a contradiction if $ c_0/M + c_1/R $ is larger than the right handside of \eqref{above}. This happens if we choose $c_1 > 4L$ and $c_0 = \delta(L+1)$. Therefore it follows that  $v\leq w$ in $\Pi \cap \overline{B_R}$. We can now conclude that 
$$
u_1- u_2 = Mv \leq c_0+c_1M/R +  2M(1 + \dfrac{\Lambda n}{\lambda} ) /R^2
 \quad \hbox { in }  \Pi\cap B_1(0).
$$
The lower bound can be obtained with above argument applied to $u_2-u_1$. 
\end{proof}

\section{Homogenization in a Strip Domain}


Let
$u_\e$  solve $(P)_{\e,\nu,\tau,q}$ given in Section 1.1. Then there is a unique linear function $v_\e$ 
such that  $v_\e = u_\e $  on  $H_{-1} \cup \{z_0\}$, where  $z_0: = \tau - \nu/2 $ is a fixed reference point in $\Pi$.  
We define the average slope  $\mu(u_\e)$ of  $u_\e$ as follows
\begin{equation}\label{average}
\mu(u_\e):=  \partial_{\nu} v_\e.
\end{equation}

\begin{theorem} \label{main-cell-2} 
	The followings hold for $u_\e$ solving $(P)_{\e,\nu,\tau,q}$:
	\begin{itemize}
		\item[(a)] For irrational directions $\nu$, there exists a unique constant $\mu= \mu(\nu, q)$ such that $u_\e$  converges uniformly to the linear profile
		$$u(x) := \mu((x-\tau)\cdot \nu + 1) + l(x).$$  The same holds for rational directions $\nu$ with $\tau=0$.
		
		
		\item[(b)] [Error estimate] There exists a constant $C>0$ depending on  $\lambda$, $\Lambda$, $n$, and the slope of $l(x)$ such that the following holds: if $\nu$ is an irrational direction or  $\nu$ is a rational direction with  $\tau=0$,  
		then 
		$$|\mu(u_\e) -\mu|  \leq C \Lambda(\e,\nu) \quad \hbox{in} \quad \Pi,$$		
		where 
		\begin{equation}\label{TheFunction}
		\Lambda(\e,\nu) = \left\{\begin{array}{lll}  \displaystyle{\inf_{0<k<1}}  \{\e^{ k} T_\nu +  \e^{1-k}\} &\hbox { if
		}& \nu \hbox{ is a rational direction} \\ \\
		\displaystyle{\inf_{0<k<1, N \in \N}} \{ \e^{ k}N  + \omega_\nu(N) +  \e^{1-k}\}    &\hbox { if
		}& \nu  \hbox{ is an irrational direction.}
		\end{array}\right.
		\end{equation}
		In (\ref{TheFunction}), 
		$T_\nu$  is as given in (a) of Lemma~\ref{lemma-M} (which is the period of $G(P,y)$ on the Neumann boundary $H_0$) and   $\omega_\nu(N)$ is as given in 
		(\ref{mode}) with  $\omega_\nu(N) \to 0 $ as $N\to \infty$.		
		\item[(c)] 
		Let $\nu$ be an  irrational direction. For any $\delta>0$,  there exist $T>0 $ and $P \in \Z^n$ such that 
		$$|T\nu -P| \leq \delta .$$
		Let $\theta= \theta( \delta, \nu)$ be the angle between $\nu$ and $P$, then $$\Lambda(\e, \nu) \leq 3 \delta \,\, \hbox{ for } \,\,  \e <\delta^2\theta.$$
		\item[(d)] Let $\nu$ be a rational direction, and let $\delta>0$. Then
		$$ \Lambda(\e, \nu) \leq 2\delta \,\, \hbox{ for } \,\, \e<\dfrac{\delta^2}{T_\nu}.
		$$
	\end{itemize}
\end{theorem}

To prove Theorem~\ref{main-cell-2} we begin with a preliminary lemma.  The following lemma states that $u_\e$ looks like a linear profile  (almost flat) on each  hyperplane normal to $\nu$.   


\begin{lemma} \label{Claim 2}
	Away from the Neumann boundary $H_{0}$, $u_\e -l(x)$ is almost
	a constant on hyperplanes parallel to $H_{0}$. More precisely,
	for 
	$x_0 \in \Pi$ we denote $$d := {\rm dist}(x_0, H_{0}) >0$$
	and $H_{d}:=  \{(x-\tau) \cdot \nu  =d\}= \{(x-x_0) \cdot \nu  =0\}$. 
	Then the followings hold:
	\begin{itemize}
		\item[(a)] If $\nu$ is a rational direction,   there exists a constant $C>0$ depending on $\alpha$, $\lambda$, $\Lambda$, $n$, and the slope of $l$ such
		that for any $x \in H_{d}$
		\begin{equation}
		|(u_\e (x)-l(x))-(u_\e(x_0)-l(x_0))| \leq C(d^{-1}+1)(T_\nu \e).
		\end{equation}
		where $T_\nu$ is a constant depending on $\nu$, given as in  (a)  of Lemma~\ref{lemma-M}.
		
		\item[(b)] If $\nu$ is an irrational direction,
		there exists a constant $C>0$ depending on $\alpha$, $\lambda$, $\Lambda$, $n$ and the
		slope of $l$ such that for any $x \in H_{d}$
		\begin{equation} \label{C2}
		|(u_\e (x)-l(x))-(u_\e(x_0)-l(x_0))| \leq C(d^{-1} \e (MN +  \omega_\nu(N))  + \omega_\nu(N)) 
		\end{equation}
		for any $N \in \N  $ and $\e>0$  with   $\e(MN + \omega_\nu(N))  <1$, where $M$ is a dimensional constant given as in  (b)  of Lemma~\ref{lemma-M}, and $\omega_\nu(N)$ is given as in (\ref{mode}).
	\end{itemize}
\end{lemma}

\begin{proof} First, we consider a rational direction  $\nu$. By (a) of  Lemma~\ref{lemma-M}, for any $x \in H_{d}$, there is $y \in H_{d}$ such
	that $|x-y| \leq T_\nu \e$ and $y=x_0$ mod $\e \Z^n$. Then by comparison
	\begin{equation} \label{inwon}
	u_\e(x) = u_\e(x + (y-x_0)) - l(y) + l(x_0).
	\end{equation}
	Hence $u_\e (x_0)= u_\e (y) - l(y) + l(x_0)$ and we get
	$$
	\begin{array} {lll}
	|(u_\e (x)-l(x))-(u_\e(x_0)-l(x_0)) |
	&\leq&
	|u_\e(x)- u_\e(y)| + |l(y)- l(x)|    \\ \\&\leq&
	|u_\e(x)- u_\e(y)| + C T_\nu \e \\ \\ &\leq& C d^{-1} T_\nu \e  +  C T_\nu \e 
	\end{array}
	$$
	where  the third inequality follows from
	Theorem~\ref{lemma-reg0}.
	
	Next,  we consider an irrational direction $\nu$  and let $x \in
	H_{d}$. By (b) of Lemma~\ref{lemma-M}, for any $N \in \N $, there exists $y \in \R^n$
	such that  $|x-y| \leq \e(MN + \omega_\nu(N))$, $y = x_0$  mod $ \e \Z^n$
	and
	\begin{equation}\label{projection}
	{\rm dist}(y, H_{d}) <
	\e \omega_{\nu}(N). 
	\end{equation}
	Observe that 
	$$|(u_\e (x)-l(x))-(u_\e(x_0)-l(x_0))| \leq 
	|u_\e(x) - u_\e (y)| + |(u_\e (y)-l(y)) - (u_\e (x_0)-l(x_0)| +|l(y) -l(x)|$$
	where, from Theorem~\ref{lemma-reg0},
	$$|u_\e(x) - u_\e (y)|  \leq    Cd^{-1} \e (MN +  \omega_\nu(N)).$$ 
	
	Next we project $y$ to $x_1\in H_{d}$ and use Lemma~\ref{perturbation} for $G_1 = G$ and $G_2(p,x) = G_2(p, x+(x_0-x_1)) = G_2(p, x+(y-x_1))$ with $\delta = \omega_{\nu}(N)$ to conclude that 
	$$ |(u_\e (x_0)-l(x_0)) - (u_\e (x_1)-l(x_1))| \leq C\omega_\nu (N).$$   
	
	and then once again use Theorem~\ref{lemma-reg0} with \eqref{projection} to compare $u(y)$ with $u(x_1)$ to conclude that 
	$$
	|(u_\e (y)-l(y)) - (u_\e (x_0)-l(x_0))| \leq C(\omega_\nu (N) +\e)
	$$
	and lastly
	$$|l(y) -l(x)| \leq C|y-x| \leq C\e(MN + \omega_\nu(N)) \leq 
	Cd^{-1} \e (MN +  \omega_\nu(N))
	$$ where the last inequality follows if $\e(MN + \omega_\nu(N))<1$.

\end{proof}

Since  $u_\e$ is flat on each hyperplanes located a constant $d$-away from the Neumann boundary, $u_\e$ can be approximated well by a linear solution as in the following corollary.  The proof of Corollary~\ref{coC2}    follows from  the comparison principle
(Theorem~\ref{thm:comp}) and Lemma~\ref{Claim 2} with $d=\e^{1-k}$.

\begin{corollary}\label{coC2}

For a solution $u_\e$  of $(P)_{\e,\nu,\tau,q}$, let $v_\e$ be the unique  linear function given as in (\ref{average}). 
	 Then  there exists  a constant $C$  depending on  $\lambda$, $\Lambda$, $n$ and the
	slope of $l$ such that for any   $N \in \N$ and $0<k<1$,
	$$
	|u_\e (x)-v_\e(x)| \leq   \left\{\begin{array}{lll}  C( \e^{ k} T_\nu +  \e^{1-k}) &\hbox { if
	}& \nu \hbox{ is a rational direction} \\ \\
	C(\e^{ k}N  + \omega_\nu(N)  + \e^{1-k})    &\hbox { if
	}& \nu  \hbox{ is an irrational direction}
	\end{array}\right.
	$$ and hence
	$$
	|u_\e (x)-v_\e(x)| \leq C \Lambda(\e, \nu) .
	$$
	
\end{corollary}

 Due to the uniform interior regularity of $\{u_\e\}$ (Theorem~\ref{lemma-reg0}), along a subsequence they locally uniformly converges to  $u$  in  $\Pi$. Let us choose one of the convergent subsequence $u_{\e_j}$ and
denote it by $u_j$, i.e., $u_j =u_{\e_j} $. Let $v_j = v_{\e_j}$ and $\mu_j = \mu(u_{\e_j})$, both as given in \eqref{average}.

\medskip

 Corollary~\ref{coC2} implies that for any $\nu \in \mathcal{S}^{n-1}$,  $\lim u_j$ is linear. More precisely, the slope $\mu_j$ converges  
as $j \to \infty$ (see Lemma 4.1 of \cite{CKL}), and hence by Corollary~\ref{coC2} $$\lim u_j = \lim v_j = \mu((x-\tau)\cdot \nu + 1) + l(x) = u$$ for $\mu := \lim \mu_j$. 

\medskip
 
Next, we prove that the subsequential limit is unique, i.e., $\mu$ does not depend on the subsequence $\{\e_j\}$ when $\nu$ is irrational or $\nu$ is rational with $\tau=0$.  We will also obtain a mode of convergence of $\mu_\e$.

\medskip

\noindent {\bf Proof of Theorem~\ref{main-cell-2} (a) and (b) for irrational directions:}  

\medskip

Let $\nu$ be an irrational direction and let $u$ be a subsequential limit of $u_\e$. We claim that  
 $$\partial u /
	\partial \nu  =\mu(\nu,q)$$ for a constant $\mu(\nu,q)$ which depends on
	$\nu$ and $q$, not on $\tau$ or the subsequence $\{\e_j\}$. More precisely, 
	\begin{equation}\label{error2}
	|\mu(u_\eta) -\mu(u_\e)| \leq C(\Lambda(\e, \nu) +\eta).
	\end{equation}

\medskip

For the proof of (\ref{error2}), let $0<\eta < \e$ be sufficiently small. Let $$w_{\e}(x) = \displaystyle{\frac{u_\e(\e
			x)}{\e}}, \,\,\,\,\,  w_\eta(x)=\displaystyle{ \frac{u_\eta(\eta
			x)}{\eta}} $$ and denote by $H^1$ and $H^2$,  the
	corresponding Neumann boundaries of $w_{\e}$ and $w_{\eta}$,
	respectively. By (c) of Lemma~\ref{lemma-M}, for $\tau \in
	\R^n$, there exist $s_1 \in H^1$ and $s_2 \in H^2$ such
	that
	$$|\tau-s_1| \leq \eta \,\,\hbox{ mod } \Z^n, \,\,\hbox{ and }\,\,|\tau-s_2| \leq \eta \,\,\hbox{ mod } \Z^n.$$
	Hence after translations by $\tau-s_1$ and $\tau-s_2$, we may suppose that
	$w_{\e}(x)$ and
	$w_\eta(x)$ are  defined  on the extended strips
	$$
	\Omega_\e:=\{x: -\frac{1}{\e} \leq (x-\tau)\cdot\nu \leq 0\}
	\quad \hbox{ and } \quad 
	\Omega_\eta:=\{x: -\frac{1}{\eta} \leq (x-\tau)\cdot\nu \leq 0\}
	$$
	respectively, with 
	$$w_\e=l_\e(x) \quad \hbox{ on } \quad \{(x-\tau)\cdot\nu = -\frac{1}{\e}\}$$ and
	$$w_\eta = l_\eta(x) \quad \hbox{ on } \quad  \{(x-\tau)\cdot\nu=-\frac{1}{\eta}\}$$ where
	$l_\e$ and $l_\eta$ are linear functions with the same slope as
	$l(x)$. Moreover on $H_{0}$ we have
	$$
	\partial w_\e /\partial \nu = G(Dw_\e ,  x-z_1 ) \quad \hbox{ and } \quad \partial w_\eta
	/\partial \nu = G(Dw_\eta ,  x-z_2 ) 
	$$ for some $|z_1|,|z_2|
	\leq \eta$.  Observe that by  H\"{o}lder continuity of $G$, i.e., by (G2) 
	\begin{equation}\label{g22}
	| G(p, x -z_1) - G(p,x-z_2) | <m(1+|p|)\eta
	\end{equation}
	
	\medskip
	
	 Let $v_\e$ be given in \eqref{average}. Then
	by Corollarly~\ref{coC2} (after a translation),
	\begin{equation}\label{above0}
	|w_\e(x)-\frac{v_\e(\e x)}{\e}| \leq \frac{C\Lambda(\e, \nu )}{\e}.
	\end{equation}
	\vspace{10pt}
	Note that
	$$
	\dfrac{v_\e(\e x)}{\e}=  \mu_\e ((x-\tau)\cdot \nu +\dfrac{1}{\e}))
	+l_\e(x).
	$$
	From \eqref{above0} and the comparison principle, it follows that
	\begin{equation}\label{above}
	(\mu_\e-C\Lambda(\e, \nu))((x-\tau)\cdot\nu  + \frac{1}{\e})\leq w_\e(x)
	-l_\e(x) \leq (\mu_\e +C\Lambda(\e, \nu))((x-\tau)\cdot\nu +\frac{1}{\e}).
	\end{equation}
	Here we denote by $l_1$ and $l_2$, the following linear profiles 
	$$
	l_1(x) = a_1(x-\tau)\cdot\nu+b_1\hbox{ and
	}l_2(x)=a_2(x-\tau)\cdot\nu+b_2,
	$$ whose respective slopes are
	$a_1=\mu_\e+C\Lambda(\e, \nu)$ and $a_2=\mu_\e - C\Lambda(\e, \nu)$.   $b_1$
	and $b_2$ are chosen so that
	\begin{equation} \label{l1l2}
	l_1 (x) =l_2 (x)= \omega_{\eta}(x) -l_\eta(x) =0 \,\,\hbox{ on
	}\,\, \{x:(x-\tau)\cdot\nu= -\frac{1}{\eta}\}.
	\end{equation}
	
\medskip
	
	 Now we define
	$$
	\overline{w}(x): = l_\eta(x) + \left\{\begin{array} {lll} l_1(x)
	&\hbox{ in } &
	\{-1/\eta \leq(x-\tau)\cdot \nu \leq -1/\e\}\\ \\
	w_\e(x)-l_\e(x)+c_1  &\hbox{ in }& \{-1/\e \leq(x-\tau)\cdot \nu \leq
	0\}
	\end{array}\right.
	$$
	and
	$$
	\underline{w}(x): = l_\eta(x) + \left\{\begin{array} {lll} l_2(x)
	&\hbox{ in } &
	\{-1/\eta \leq(x-\tau)\cdot \nu \leq -1/\e\}\\ \\
	w_\e(x)-l_\e(x) +c_2  &\hbox{ in }& \{-1/\e \leq(x-\tau)\cdot \nu \leq
	0\}
	\end{array}\right.
	$$
	where $c_1$ and $c_2$ are constants satisfying $$l_1=w_\e-l_\e+c_1 =c_1 \hbox{ and } l_2=w_\e -l_\e+c_2=c_2$$
	on
	$\{(x-\tau)\cdot\nu= -1/\e\}$. (See Figure 2.)
	Note that by \eqref{l1l2}, 
	$$ \underline{w} =\overline{w} = w_\eta \quad \hbox{ on } \quad \{x:(x-\tau)\cdot\nu= -\frac{1}{\eta}\}$$
	and also due to \eqref{above},  
	$$
	\overline{w}(x)= l_\eta(x)+ \min (l_1(x), w_\e(x) -l_\e (x)+ c_1)$$ and $$\underline{w}(x) = l_\eta(x)+ \max ( l_2(x), w_\e(x)  -l_\e (x)+c_2)
	$$ 
	in $\{-\frac{1}{\e}\leq (x-\tau)\cdot\nu \leq 0\}$. 
	Thus it follows that $\overline{w}$ and $\underline{w}$ are respectively viscosity super- and subsolution of (P). Hence we obtain 
	\begin{equation} \label{g33}
	\underline{w} \leq \tilde{w}_\eta \leq \overline{w} 
	\end{equation}
	where $\tilde{w}_\eta$ is a solution of (P) in $\Omega_\eta$ with $\tilde{w}_\eta = w_\eta = l_\eta(x) $ on $\{(x-\tau)\cdot \nu = -1/\eta\}$, and $\partial \tilde{w}_\eta /\partial \nu = G(D\tilde{w}_\eta,  x-z_1)$ on $H_{0}$. Then by (\ref{g33}) and Lemma~\ref{perturbation} with (\ref{g22}), 
	$$
	|\mu_\eta -\mu_\e| \leq |\mu_\eta -\mu(\tilde{w}_\eta)| + |\mu(\tilde{w}_\eta) -\mu_\e| \leq  C(\Lambda(\e,\nu) +\eta)
	$$ 
	where $\mu(\tilde{w}_\eta)$ is  the slope of the linear approximation of $\tilde{w}_\e$. The above inequality implies that the slope $\mu$ of a subsequential limit of $u_\e$ depends on neither the subsequence $\{\e_j\}$ nor $\tau$. Also  sending $\eta \to 0$, we get an error estimate (d) when $\nu$ is irrational.

	\medskip

\noindent{\bf Proof of Theorem~\ref{main-cell-2} (a) and (b) for rational directions:}
	Let $\nu$ be a rational direction with $\tau =0$. We claim that $\partial u /
	\partial \nu  =\mu(\nu,q)$ for a constant $\mu(\nu,q)$ which depends on
	$\nu$ and $q$, not on the subsequence $\{\e_j\}$. More precisely, if $\eta \leq
	\e$, then 
	\begin{equation}\label{error222}
	|\mu(u_\eta) -\mu(u_\e)| \leq C\Lambda(\e, \nu).
	\end{equation}

\medskip

The proof of (\ref{error222}) is parallel to that of  (\ref{error2}). 
	Let $w_\e$ and $w_{\eta}$ be as given in the proof of  (\ref{error2}). Note that since
	$\Omega_\e$ and $\Omega_\eta$ have their Neumann boundaries passing
	through the origin, $\partial w_\e /\partial \nu =G(x) =\partial
	w_\eta /\partial \nu$ without translation of the $x$ variable, and
	thus we do not need to use the properties of hyperplanes with an
	irrational normal (Lemma~\ref{lemma-M} (b)) to estimate the error
	between the shifted Neumann boundary datas. In other words, there
	exist $q_1 \in H^1$ and $q_2 \in H^2$ such that
	$p=q_1=q_2$ mod $\Z^n$, hence $G(\cdot, x-z_1) = G(\cdot,  x-z_2)$   in the proof of  (\ref{error2}). Following the proof of  (\ref{error2}), we
	get an upper bound $\Lambda(\e, k)$ of $|\mu_\eta-\mu_\e|$. Note that
	we do not have the term $\eta$ in (\ref{error222}) since
	$G(\cdot,  x-z_1) = G(\cdot,  x-z_2)$. 
Sending $\eta \to 0$ in (\ref{error222}), we obtain the error estimate (b) for rational directions with $\tau =0$.

\medskip

\medskip

\noindent 
\textbf{Proof of Theorem~\ref{main-cell-2} (c) and (d):} Let $\delta>0$ and let  $\nu$ be an irrational direction. Lemma~\ref{FK} implies that there is a positive number $T_\nu(\delta) \leq \delta^{-(n-1)}$ such that $|T_\nu(\delta)\nu| \leq \delta$ mod $\Z^n$. Then for some $ P \in \Z^n$ and $T = T_\nu(\delta) + O(\delta)$ 
$$
|T \nu -P|\leq \delta
$$
and 
$T \nu \in P+<\vec{P}>^\perp $. Let $\theta=\theta(\delta, \nu)>0$ be the angle between $\nu$ and $\vec{P}$, then 
\begin{equation} \label{lalalala}
|T \nu -P | =T \theta \leq \delta
\end{equation} 
and we can observe the set $\{m T \nu \,\,|\,\, 0\leq m \leq [\displaystyle{\frac{1}{T\theta}}]\} $ is evenly distributed  mod $\Z^n$. Hence we get
\begin{equation} \label{lalalalala}
\omega_{\nu}(N) \leq T\theta  \hbox{ when } N =  [\frac{1}{\theta}] .
\end{equation} 

Let $ \e(\delta, \nu)$ be a  constant depending on $\delta$ and the direction $\nu$ such that 
\begin{equation} \label{lalala}
\e(\delta, \nu)= \delta^{2} \theta = \delta^{2} \theta(\delta, \nu).
\end{equation}
Then for $0<\e< \e(\delta, \nu) $,
$$
\Lambda (\e, \nu) =  \displaystyle{\inf_{0<k<1, N \in \N}} \{ \e^{ k}N  + \omega_{\nu}(N) +  \e^{1-k}\}
\leq
\displaystyle{\inf_{0<k<1}} \{ \e^{ k}/\theta  + T\theta +  \e^{1-k}\} \leq  \displaystyle{\inf_{0<k<1}} \{ \e^{ k}/\theta  +  \e^{1-k}\} + \delta
$$ where the first and last inequalities follow from (\ref{lalalalala}) and (\ref{lalalala}) respectively. 
Then   by (\ref{lalala})
$$\displaystyle{\inf_{0<k<1}} \{ \e^{ k}/\theta  +  \e^{1-k}\} \leq \displaystyle{\inf_{0<k<1}} \{ (\delta^2\theta)^{ k}/\theta  +   (\delta^2\theta)^{1-k}\}.$$
The infimum is taken when $0<k=\ln(\theta \delta)/\ln (\theta \delta^{2} )<1$  and 
$$\displaystyle{\inf_{0<k<1}} \{ (\delta^{2}\theta)^{ k}/\theta  +   (\delta^{2}\theta)^{1-k}\} = 2 \delta.$$
Hence we can conclude $\Lambda (\e, \nu) \leq 3 \delta$ for $\e < \e(\delta, \nu)=\delta^2 \theta$.

\medskip

Next, we consider a rational direction  $\nu$.  For  $\delta>0$, let $\e < \delta^2/T_\nu$. Then we can check
$$\Lambda(\e, \nu)=\displaystyle{\inf_{0<k<1}} \{ \e^{ k}T_\nu  +  \e^{1-k}\} \leq \displaystyle{\inf_{0<k<1}} \{ \delta^{2k} T_\nu^{1-k}  +   \delta^{2(1-k)}T_\nu^{k-1}\} = 2 \delta.$$ 

\hfill$\Box$

The following lemma will be used in the next section. 

\begin{lemma} \label{Claim 22}
	Let $\nu=e_n$, $\tau=0$, and let $w$ solve
	$$
	\left\{\begin{array} {lll} F (D^2 w, x/\e)=0 &\hbox{ in } &  \{-N\e \leq x_n \leq 0\}; \\ 
	\partial w /\partial x_n =G(Dw,  x/\e) &\hbox{ on }& H_0;\ \\
	w = A  &\hbox{ on }& H_{-N\e}.
	\end{array}\right.
	$$
	where $N$ and $A$ are constants. Then there is a constant $C=C(\lambda, \Lambda, n)$ such that 
	
	$$|w(x)-w(x_0)|\leq  C \e  \quad \hbox{ for  } \quad x, x_0 \in H_{-N\e}.$$

\end{lemma}

\begin{proof} 
	For $x_0,x \in H_{-N\e}$, choose $y \in H_{-N\e}$ such that $|x-y| \leq  \e$ and $y=x_0$ mod $\e
	\Z^n$. Observe that then  $w(y) = w(x_0)$, since $G$ is $1$-periodic on
	$H_0$. Therefore
	$$
	|w (x)-w(x_0)) |=
	|w(x)- w(y)|   \leq C \|w-A\|_{L^\infty}|\frac{x-y}{N\e}| \leq C \e,
	$$
	where the second inequality is from the interior Lipschitz regularity (Theorem~\ref{lemma-reg0}) applied to $w(N\e x)$.	
\end{proof}

\section{Continuity over normal directions}

In the previous section we have shown that for an irrational direction $\nu \in \mathcal{S}^{n-1} -
\RR\ZZ^n$, there is a unique homogenized slope $\mu(\nu,q)$ for any solution $u^\nu_\e$ of $(P)_{\e,\nu, \tau, q}$ in $\Pi(\nu,\tau)$. In this section we investigate the continuity properties of $\mu$ with respect to $\nu$ and $q$, as well as the mode of convergence for $u_\e^{\nu}$ as the normal direction $\nu$ of the domain varies. 

\medskip

We first show that $\mu$ is Lipschitz with respect to $q$, which directly follows from the $1$- homogeneity of $G$.

\begin{theorem}\label{lipinq}
	For $\nu \in \mathcal{S}^{n-1} - \RR\ZZ^n$,  $\mu(\nu,q)$ is uniformly Lipschitz in $q\in <\nu>^{\bot}$, independent of $\nu$.
\end{theorem}

\begin{proof}
	
	For $q_1, q_2 \in <\nu>^{\bot}$, let $u_\e^i$ be the unique bounded solution of $(P)_{\e, \nu,\tau, q_i}$ for $i=1,2$. Let $m$ be the Lipschitz constant for $G$ given in (G1) and $c$ be as given in (G3). Then  it follows that 
	$$
	w_{\pm}(x):= u^1_\e(x) + (q_2-q_1)\cdot x \pm \frac{m}{1-c} |q_1-q_2|(x\cdot\nu)
	$$
	is respectively a super and subsolution of $(P)_{\e, \nu, \tau,  q_2}$. Hence by Corollary~\ref{thm:cp2} we have 
	$$
	w_- \leq u^2_\e  \leq w_+.
	$$  
	
	From here and Theorem~\ref{main-cell-2} it follows that 
	$$
	|\mu(\nu, q_1) - \mu(\nu, q_2) | \leq \frac{m}{1-c}|q_1-q_2|.
	$$
	
\end{proof}

The dependence of  $\mu$ on $\nu$ is a much more subtle matter due to the change of the domain and the resulting changes in boundary conditions on the Neumann boundary. From now on we work with a fixed choice of $q$ and denote $\mu= \mu(\nu)$.

\medskip

For $s\geq 0$,  let $T_\nu(s)$ be the smallest positive number $\geq 1$ 
such that $$|T_\nu(s) \nu| \leq s \,\,\hbox{ mod }\,\,
\ZZ^n. $$  Note that with this definition $T_\nu$ given in the Appendix corresponds to $T_{\nu}(0)$ which is larger than all $T_\nu(s)$. In general  Lemma~\ref{FK} yields
\begin{equation} \label{lastt}
T_{\nu}(s) \leq  \sqrt{n} \cdot s^{-(n-1)}.
\end{equation}

\medskip

\begin{theorem}\label{continuity}
	With fixed $q$, let us denote $\mu=\mu(\cdot,q): (\mathcal{S}^{n-1} - \RR\ZZ^n) \to \RR$ be as given
	in Theorem~\ref{main-cell-2}. Then $\mu$ has a continuous extension $\bar{\mu}(\nu):
	\mathcal{S}^{n-1} \to \RR$. More precisely, let us fix a direction $\nu \in
	\mathcal{S}^{n-1}$ and a constant $\delta>0$. If $\nu_1$ and $\nu_2$ are irrational directions such that
	
	\begin{equation}\label{small}
	0< \theta_i := |\nu_i-\nu| < \frac{\delta^{5/2}}{T_\nu(\delta^{5/2})} \,\,\,\hbox{ for }i=1,2
	\end{equation}
	
	then we have 
	\begin{itemize}
		\item[(a)]$|\mu(\nu_1)-\mu(\nu_2)|<C \delta^{1/2}$.

		\item[(b)] $\bar{\mu}(\nu)$ is H\"{o}lder continuous on $ \mathcal{S}^{n-1}$ with  a H\"{o}lder exponent of $\dfrac{1}{5n}$. 
	\end{itemize}
\end{theorem}

\medskip

\begin{remark}
	In the proof we indeed show that, for  any directions $\nu_1$ and $\nu_2$ satisfying \eqref{small}, the range of $\{\mu(u_\e^{\nu_i})\}_{\e, i}$ fluctuates only by $\delta$, if $\e$ is sufficiently small . The fact that $\nu_i$'s are irrational is only used to guarantee that there is only one subsequential limit for $\mu(u_\e^{\nu_i})$. 
	\end{remark}

\medskip

For the rest of the paper we prove (a) of Theorem~\ref{continuity}. Theorem~\ref{continuity} (b) follows from (\ref{lastt}), (\ref{small}) and  Theorem~\ref{continuity} (a).

\medskip

\medskip

\subsection{Basic settings and Sketch of the proof } \label{4.1}

For notational simplicity and clarity in the proof, we assume that $n=2$ and $\nu =e_2$. We will explain in the paragraph below how to
modify the notations and the proof for $\nu\neq e_2$. For general dimension $n$, we refer to Remark~\ref{generaln}. We denote
$$
\Pi :=\Pi(e_2,0) \hbox{ and } \Pi^{\nu_i}:= \Pi(\nu_i,0), \hbox { for } i=1,2.
$$
We also denote 
$$
H_0=H_0(e_n), \quad H_0^{\nu_i} :=H_0(\nu^i) \quad \hbox{ for }i=1,2.
$$
For given 
$$m \in \N  \hbox{ and } \delta:= 1/m>0,
$$ we divide the unit strip $\R\times [0,1]$ by $m$ number of small horizontal strips of width $\delta$ 
and define a family of functions $\{G_k\}_k$ so that the value of $G_k$ at $(x_1, x_2)$ is  same as the value of $G$ at $(x_1, \tilde{x}_2)$, where $(x_1, \tilde{x}_2)$ is the projection of $(x_1,x_2)$ onto  the bottom of the k-th strip. 
More precisely we define

\begin{equation}\label{neumann:proj}
G_k(x_1, x_2):=G(x_1, \delta(k-1)) \hbox{ for }
k=1,...,m.
\end{equation} Then  $G_k$ is a $1$-periodic function with respect to $x_1$.



\medskip

Next we introduce the parameters
\begin{equation}\label{angles}
\theta_1:=|\nu_1-e_2|,  \,\,\, \theta_2:=|\nu_2-e_2|
\end{equation}
and 
\begin{equation}\label{steps}
N :=[\dfrac{\delta}{\theta_1}], \,\,\, M: =[\dfrac{\delta}{\theta_2}].
\end{equation}
Without loss of generality,  assume $\theta_2 \leq \theta_1$ and thus $N\leq M$. 

\medskip

If $\theta_i$'s are sufficiently small, then we will be able to approximate $G$ on both of the Neumann boundary $H_0^{\nu_1}$ and $H_0^{\nu_2}$ using the universal boundary data $G_k$'s  which depends only on $\delta$,  but not on the direction $\nu_1$ nor $\nu_2$. In particular, in meso-scopic scale $G$ can be approximated by many repeating pieces ($N$ for $\nu_1$ and $M$ for $\nu_2$) of $G_k$, for a suitable $1\leq k \leq m$ on each pieces of $H_0^{\nu_i}$. Thus the problem already experiences averaging phenomena: we call this as {\it the first} or {\it near-boundary} homogenization.  Note that  in this step the only difference in the averaging phenomena  between  the two directions $\nu_1$ and $\nu_2$, besides the errors in terms of $G$ and $G_k$ on $H_0^{\nu_i}$, is the number of repeating data $G_k$ for each $k$. This explains the proximity of $\mu(\nu_1)$ and $\mu(\nu_2)$.

\medskip

On the other hand,  since $\nu_i's$ are irrational directions, the distribution of $G_k$ approximates the given $G$ on $H_0^{\nu_i}$ in large scale.  Since $\nu_1$ and $\nu_2$ are close to the rational direction $e_2$,  the averaging behavior of a solution $u_\e^{\nu_i}$ in
$\Pi^{\nu_i}$  would appear in a very large scale, in other words only after $\e$ gets very
small. We call this as the {\it secondary} homogenization.
\medskip

The two-scale homogenization procedure has been introduced in \cite{CK}, \cite{CKL}. It allows studying continuity properties of the homogenized boundary data as we approach the rational direction, which might be singular points as described in the introduction. This point of view was also employed in \cite{F} and \cite{FK} to study homogenization for general operators, by studying the singularity of homogenized operator at rational directions. Let us also point out near the boundary  the small-scale oscillation of the operator interacts with that of boundary data to create a meso-scale averaging phenomena. Due to this interaction, characterizing the homogenized boundary condition remains a challenging and interesting open problem. After the first homogenization, the boundary data changes to periodic data  in a meso-scale (which will be $N\e$ below), and hence the operator is well approximated by the homogenized operator $\bar{F}$ in the second homogenization in large scale.



\medskip

Below we begin the analysis of the two-step homogenization as described above. We will work with small $\e>0$ satisfying
\begin{equation}\label{small2}
\e \leq \frac{\delta\theta_i}{T_\nu(\delta^{5/2})} \,\,\,\hbox{ for }i=1,2
\end{equation}
which can be stated as 
\begin{equation}\label{order202}
0<\e \leq\delta\theta_i \,\,\, \hbox{ for }i=1,2
\end{equation}
since $T_\nu (s) \equiv 1$  when $\nu=e_2$.
It follows that 
\begin{equation} \label{that}
mN\e \leq mM\e \leq \delta.
\end{equation}
After the near-boundary homogenization, $u^{\nu_1}_\e$ will be approximated by a solution which has periodic boundary data  with period $mN\e$.  With \eqref{that}  it follows that $u_\e^{\nu_1}$  fluctuates in order of $\delta$ in the interior of the strip domain.

\medskip

 On the other hand, (\ref{small}) of Theorem~\ref{continuity} can be stated as 
\begin{equation}\label{parameter}
0< \theta_1 \leq \delta^{5/2}.
\end{equation}
It follows then that
\begin{equation} \label{this}
 1/N\leq \delta^{3/2}
\end{equation}
which ensures $u_\e^{\nu_i}$ to homogenize $N\e$-close to the Neumann boundary.

\medskip
 
 Next define 

\begin{equation}\label{interval_k}
I_k:=[(k-1)N\e,kN\e] \times \RR \,\,\hbox{  for }\,\, k \in \ZZ.
\end{equation}
Then we can observe that in each $I_k$, 
the Neumann boundary  $H_0^{\nu_1}$  is located within $\delta \e$-distance from   $H_0 + \delta \e (k-1)e_2$,  mod $\e\Z^n$.
 Thus on each  $H_0^{\nu_1} \cap I_k$, $G$ is approximated well by $G_k$  for $1\leq k \leq m$. If we extend the definition of $G_k$ over $k \in \Z$ by letting $G_k = G_{\bar{k}}$ for  $k=\bar{k}$ (mod $m$) then we have
\begin{equation}\label{eqn1}
|G(p,  \frac{x}{\e}) - G_{k}  (p,  \frac{x}{\e})| <
C(1+|p|)\delta \hbox{ on } H^{\nu_1}_0 \cap I_k \hbox{ for  } k\in\ZZ.
\end{equation}
Similarly for $\nu_2$, if we define $J_k:= [(k-1)M\e, kM\e]\times\R$ for $k\in\Z$.

\medskip

\begin{remark} For $\nu \neq e_2$ in $\R^2$, there exists a rational
	direction $\tilde{\nu}$ such that for $T = T_\nu(\delta^{5/2})$, 
	$$T \tilde{\nu} =0 \,\, (\hbox{mod }\ZZ^2); \quad  
	|\nu-\tilde{\nu}|\leq \delta^{5/2}/T.$$
	Observe that if Theorem~\ref{continuity} holds for the rational
	direction $\tilde{\nu}$, it also holds for $\nu$. For the proof of
	the theorem for $\tilde{\nu}$, let $x' = x-(x\cdot\tilde{\nu}) \tilde{\nu}$ and  define
	$$G_k = G_k( x', x-x')= G(x',  \delta(k-1) \tilde{\nu}) \,\, \hbox{ for }\,\,1 \leq
	k \leq m.$$  Then $G_k$ is a periodic function on $\{x\cdot \tilde{\nu}=0\}$
	with a period  of $T$. The only difference between the case of
	$\tilde{\nu}$ and $e_2$ is in the periodicity of the function
	$G_k$, and it does not make any essential difference in the proof.
	we point out that instead of the conditions (\ref{parameter}), (\ref{this}) and
	(\ref{that}), we will need
	$$
	\frac{1}{TN} \leq \delta^{3/2} \hbox {;  } \quad  T\theta_1 \leq \delta^{5/2}\hbox {;  } \quad mTM\e \leq \delta
	$$
	since $G_k$ has a period of  $T$. These conditions will be ensured if
	$\theta_i$ and $\e$ satisfy the assumptions as in
	Theorem~\ref{continuity}. 
\end{remark}

\subsection{ Proof of Theorem~\ref{continuity}} \label{continuityy} 

In the first three steps we follow the heuristics above and replace the Neumann condition with the locally projected boundary data $G_k$.  Then we go through the two-step homogenization procedures to obtain the first slope $\mu^N(G_k)$ on each $I_k$ near the boundary, and then  the global slope $\mu(\nu_1)$.  While the actual first homogenization takes place in $\Pi^{\nu_1}$ , it turns out that its value has a small difference from $\mu^N(G_k)$ taken in  $\Pi$ (see Lemma~\ref{important}). This fact is important in establishing a universal domain for both directions $\nu_1$ and $\nu_2$. In fact, we rotate  the middle and inner regions to compare the slopes in $\Pi^{\nu_1}$ and $\Pi^{\nu_2}$. For this, we use the rotational  invariance of the homogenized operator $\bar{F}$. (See Lemma~\ref{Lemma4.6} and Lemma~\ref{Lemma4.7}.) 
The rest of steps are to verify that indeed $\mu(\nu_1)$ is the correct averaged slope for the problem $(P)_{\e,\nu_1, \tau, q}$.

\medskip

\noindent {\bf Step 1. First homogenization near Boundary ($N\e$ - away from $H_0^{\nu_1}$)}

\medskip
We proceed to discuss the first homogenization. 
Denote $x = (x_1, x_2)$ throughout this section. For a  given linear function $l(x) = l(x_1)$ and $k\in \ZZ$, let $u= u^{N,\e}$ and $v_k = v_k^{N,\e}$ solve the following problem with $u=l$ on $H_{-N\e}^{\nu_1}$ and $v_k=l(x)$ on $H_{-N\e}$:

\begin{equation}
\left\{\begin{array}{lll}
F(D^2 u, x/\e)=0 & \hbox{ in }& \{-N \e \leq x \cdot \nu_1 \leq 0\};\\ \\
\dfrac{\partial u}{\partial\nu_1}
(x)= G(Du,  x/\e) &\hbox{ on }&  H_0^{\nu_1} 
\end{array}\right.
\end{equation}
and
\begin{equation}\label{vk} \left\{\begin{array}{lll}
F(D^2 v_k, x/\e) =0 &\hbox{ in }& \{-N\e \leq x_2 \leq 0\};\\ \\
\dfrac{\partial v_k}{\partial x_2 }(x) =
G_k(Dv_k,  x/\e   ) &\hbox{ on }& H_0.
\end{array}\right.
\end{equation}
\medskip

\begin{definition}\label{averaged_slope}
	For a given function $u: \{-N \e\leq x\cdot \nu \leq 0\}\to \R$ and $I_k$ given as in \eqref{interval_k}, let $a_k$ and $b_k$ be the  middle points of $I_k\cap H_{-N\e/2}$ and   $I_k\cap H_{-N \e}$ respectively, and consider the unique linear function $h$ given by 
	$h = u$ at $x= a_k, b_k$ and  $D_Th(b_k) =D_Tu(b_k) 
	$.  
	(Here  $D_Th$ denotes the tangential derivative of $h$ along the direction $\nu^{\perp}$.)  Then $\mu_k(u)$ is defined by
	$$\mu_k(u):=\partial h /\partial \nu .$$ 
\end{definition}
Note that the Neumann boundary data of $v_k$ is $G_k$ on each boundary pieces $H_0 \cap I_i$   ($i \in \Z$), and hence $\mu_i(v_k)= \mu(v_k)$. For $N$ as given in \eqref{steps}, we denote     
\begin{equation}\label{slope of v_k}
\mu^N(G_k) := \mu(v_k).
\end{equation}

\begin{lemma} \label{important} 
	For $k \in \Z$ and $\mu_k(u)$ as given in Definition~\ref{averaged_slope},	
	\begin{equation}\label{eqn:2}
	|\mu_k(u) - \mu^N(G_k)| < C \delta^{1/2}.
	\end{equation}

\end{lemma}
\begin{proof}
	We will prove the lemma for $k=1$, i.e., we will compare $\mu_1(u)$ with $\mu(v_1)$. Let  $\tilde{u}$ and  $\tilde{v}_1$ solve the following problem with
	$\tilde{u}=l$ on $H_{-\e /\delta }^{\nu_1}$ and $\tilde{v}_1=l$ on $H_{-\e/\delta}$:
	$$
	\left\{\begin{array}{lll}
	F(D^2 \tilde{u}, x/\e)=0 & \hbox{ in }& \{- \e/\delta \leq x \cdot \nu_1 \leq 0\};\\ \\
	\dfrac{\partial \tilde{u}}{\partial\nu_1}
	(x)= G(D\tilde{u}, x/\e) &\hbox{ on }&  H_0^{\nu_1} 
	\end{array}\right.
	$$
	and
	$$ \left\{\begin{array}{lll}
	F(D^2 \tilde{v}_1, x/\e) =0 &\hbox{ in }& \{-\e/\delta \leq x_2 \leq 0\}\\ \\
	\dfrac{\partial \tilde{v}_1}{\partial x_2 }(x) =
	G_1(D\tilde{v}_1,  x/\e   ) &\hbox{ on }& H_0 .
	\end{array}\right.
	$$
	We will compare both of $\tilde{u}(x)$ and  $\tilde{v}_1(x)$ to $w_1(x)$ in the ball $|x| \leq \delta^{-1-\alpha_0} \e$, where $\alpha_0= 1/2$. For computational convenience we will call this number as $\alpha_0$. Let $w_1(x)$ solve $w_1 = l$ on $H^{\nu_1}_{-\e/\delta}$ with
	\begin{equation} \label{ww1}
	\left\{\begin{array}{lll}
	F(D^2 w_1, x/\e)=0 & \hbox{ in }& \{- \e/\delta \leq x \cdot \nu_1 \leq 0\};\\ \\
	\dfrac{\partial w_1}{\partial\nu_1}
	(x)= G_1(Dw_1,  x/\e) &\hbox{ on }&  H_0^{\nu_1}.  
	\end{array}\right.
	\end{equation}    
	Here observe that in the  ball  $|x| \leq \delta^{-1-\alpha_0} \e$, the hyperplanes $H_0^{\nu_1}$ and $H_0$
	only differ by $\theta_1 \delta^{-1-\alpha_0} \e $.  
	
	Below we derive some properties of $w_1$. Consider 
	$$
	\bar{w}(x):= \e^{-1} w_1(\e x).
	$$ Then by Theorem~\ref{thm:reg2},  $\bar{w}$ is  $C^{1,1}$ regular up to the Neumann boundary in a unit ball, if $\bar{w}$ has a bounded  oscillation  in the ball $|x| \leq 1/\delta$. Observe that   $( \e/\delta)^{-1} w_1(\e x/\delta)$ is defined in the strip $\{ - 1 \leq x \cdot \nu_1 \leq 0\}$ and it has a periodic Neumann data $G_1(\cdot, \cdot, x/\delta)$ with period $\delta$. Since it has a periodic boundary data, it corresponds to the case of rational direction with Neumann boundary passing through the origin. Hence 
	we can use the error estimate Theorem~\ref{main-cell-2} (b)   for the rational direction passing through the origin, with  $T_\nu =1 $. Then we obtain 
	
	\begin{equation}\label{comp00}
	|( \frac{\e}{\delta})^{-1} w_1(\frac{\e x}{\delta}) - h(x)| \leq  \displaystyle{\inf_{0<k<1}}C(\delta^{ k}  + \delta^{1-k}) =C\delta^{1/2}
	\end{equation}
	where $h$ is a linear solution approximating  $( \e/\delta)^{-1} w_1(\e x/\delta)$.  
	Then by (\ref{comp00}) 
	\begin{equation}\label{linear} | w_1(\frac{\e x}{\delta}) -\frac{\e}{\delta} h(x)| \leq  C \delta^{-1/2} \e
	\end{equation}
	and hence the oscillation of $\bar{w}$ becomes  less than $C \delta^{-1/2}$ in the ball $|x| \leq 1/\delta$. Later in the proof we will use $C^{1,1}$ regularity of $\bar{w}$ as well as the  linear approximation \eqref{linear} of $w_1$.

	\medskip

	First, we compare $\tilde{u}$ to $w_1$ in $B_{\delta^{-1-\alpha_0} \e}(0)$. For this,  we compare the boundary data of  $\tilde{u}$, that is  $G$, to $G_1$.  Observe that if $x \in H^{\nu_1}_0 \cap B_{\delta^{-1-\alpha_0} \e}(0)$, then $x \in I_k$ for some $|k|  \leq \delta^{-1-\alpha_0}/N = \delta^{-2-\alpha_0}\theta_1$. Hence for $x \in H^{\nu_1}_0 \cap B_{\delta^{-1-\alpha_0} \e}(0)$ (i.e., for $x \in H^{\nu_1}_0 \cap I_k$ with   $|k|  \leq  \delta^{-2-\alpha_0}\theta_1$),
	
	\begin{equation}\label{error00}
	\begin{split}
	|G(p, x/\e)-G_1(p, x/\e)| &\leq 
	|G(p, x/\e)-G_k(p, x/\e)|  +|G_k(p, x/\e)-G_1(p, x/\e)| \\  &\leq C[(1+|p|)\delta + (1+|p|)|k-1|\delta] \\ 
	&\leq C(1+|p|)(\delta + 
	(\theta_1 \delta^{(-1-\alpha_0)}))
	\\ &\leq  C(1+|p|)(\delta + 
	\delta^{(3/2-\alpha_0)})
	\\& \leq  C(1+|p|) \delta,
	\end{split}
	\end{equation} where the second inequality follows from (\ref{eqn1}) and the construction of $G_k$,  third inequality follows from $|k|  \leq  \delta^{-2-\alpha_0}\theta_1$,  the fourth inequality follows from (\ref{parameter}), and the last inequality follows since $\alpha_0 \leq 1/2$.
	This implies, by Lemma~\ref{perturbation},
	\begin{equation}
	\label{thisthis}
	|\tilde{u}(x) - w_1(x)| \leq C(\delta + \delta^{\alpha_0} ) (x \cdot \nu_1 + \frac{\e}{\delta}) \leq C\delta^{\alpha_0}  (x \cdot \nu_1 + \frac{\e}{\delta})\hbox{ in } |x| \leq \delta^{-1-\alpha_0}\e.
	\end{equation}
	Observe that  \eqref{linear} and \eqref{thisthis}  yield 
	$$
	|\tilde{u}(x) - L_1(x)| \leq C( \delta^{\alpha_0} +\delta^{1/2})  (x \cdot \nu_1 + \frac{\e}{\delta})  \leq C \delta^{\alpha_0}   (x \cdot \nu_1 + \frac{\e}{\delta})  \hbox{ in } |x| \leq \delta^{-1-\alpha_0}\e.
	$$
	where $L_1(x) = l(x)+ \mu(w_1) (x \cdot \nu_1+\frac{\e}{\delta})$, and  $\mu(w_1)$ is the average slope of $w_1$. In other words, we obtain  
	\begin{equation}\label{EQ1}
	|\mu_1(\tilde{u}) - \mu(w_1)|\leq C  \delta^{\alpha_0}  . 
	\end{equation}
	
	\medskip
	
	Next, we compare $\tilde{v}_1$ and $w_1$ and prove 
	$$
	| \mu(\tilde{v}_1)-\mu(w_1) |\leq C\delta^{\alpha_0}.
	$$ Recall that the oscillation of $\bar{w}$ is  less than $C \delta^{-1/2}$ in the ball $|x| \leq 1/\delta$ (see \eqref{linear}). 
	If we consider $\tilde{w} = \delta^{1/2} \bar{w}$, then this function solves the boundary condition 
	$$
	\partial\tilde{w} /\partial \nu = \tilde{G}(D\tilde{w},  x)=\delta^{1/2} G( \delta^{-1/2}D\tilde{w},  x) ,
	$$
	which satisfies the assumptions for the $C^{1,1}$ regularity theory, Theorem~\ref{thm:reg2}. Thus we have $$\|\bar{w}\|_{C^{1,1}(B_1)} \leq O(\delta^{-1/2}).$$ 
	
	For $x$ in the  $\sigma\e$-neighborhood of $H^{\nu_1}_0$, choose $\tilde{x}$ to be the closest point to $x$ on $H_0$. Then  by (G1) and (G2) with the $C^{1,1}$ regularity of $\bar{w}$ given above,  $w_1$ satisfies on $H_0$,

	$$
	|G(Dw_1(x),  \frac{x}{\e}) - G(Dw_1(\tilde{x}),  \frac{\tilde{x}}{\e})|  \leq O(\delta^{-1/2 } \sigma) (1+|Dw_1(x)|)
	$$
	Recall that the Neumann boundaries of $w_1 $ and $v_1$    
	($H^{\nu_1}_0$ and $H_0$) 
	only differ in the ball $|x| \leq \delta^{-1-\alpha_0} \e$,  by  $\theta_1 \delta^{-1-\alpha_0} \e  \leq \delta^{3/2-\alpha_0} \e $  (see (\ref{parameter})). So putting $\sigma =\delta^{3/2-\alpha_0} $,  
	$$
	|G(Dw_1(x),  \frac{x}{\e}) - G(Dw_1(\tilde{x}),  \frac{\tilde{x}}{\e})|  \leq O(\delta^{1-\alpha_0})(1+|Dw_1(x)|) \hbox{ on } H_0
	$$
	and  Lemma~\ref{perturbation} yields that  in $ |x| \leq \delta^{-1-\alpha_0}\e
	$,
	$$
	|(\tilde{v}_1 - w_1)(x) | \leq C(\delta^{1-\alpha_0} + \delta^{\alpha_0} ) (x_n + \frac{\e}{\delta})  \leq C \delta^{\alpha_0}  (x_n + \frac{\e}{\delta})  .$$
	This and \eqref{linear} yield that in $|x| \leq \delta^{-1-\alpha_0}\e$, 
	$$
	|\tilde{v}_1(x) - L(x)| \leq C (  \delta^{\alpha_0} + \delta^{1/2}) (x_n + \frac{\e}{\delta}) \leq C   \delta^{\alpha_0}  (x_n + \frac{\e}{\delta}) 
	$$
	where $L(x) = l(x)+ \mu(w_1) (x_2+\frac{\e}{\delta})$.  In other words, we obtain
	\begin{equation} \label{com1}
	|\mu(w_1) -\mu(\tilde{v}_1)| \leq  C  \delta^{\alpha_0}.
	\end{equation}
	Recalling $\alpha_0=1/2$ we conclude from (\ref{EQ1}) and (\ref{com1}) that 
	\begin{equation} \label{56}
	|\mu_1(\tilde{u}) -\mu(\tilde{v}_1)| \leq C \delta^{1/2}. 
	\end{equation}
	
	In the rest of proof, we will show $$|\mu(v_1) -\mu(\tilde{v}_1)|,\,\, |\mu_1(u) -\mu_1(\tilde{u})| \leq C \delta^{1/2}.$$
	Then the above inequalities and (\ref{56}) would imply  
	$$|\mu_1(u) -\mu(v_1)| \leq |\mu_1(u) -\mu_1(\tilde{u})| + |\mu_1(\tilde{u})-\mu(\tilde{v}_1)|+|\mu(\tilde{v}_1) -\mu(v_1)|   
	\leq C\delta^{1/2}.$$
	
	First, observe that $v_1$ and $\tilde{v}_1$ have periodic Neumann data $G_1$ on  $H_0$. Hence  by similar arguments as in the proof of (\ref{error2}),
	\begin{equation} \label{sima}
	|\mu(v_1) - \mu(\tilde{v}_1)|\leq C(\Lambda(\delta, e_2) + N^{-1}) \leq C(\delta^{1/2} + N^{-1}) \leq C \delta^{1/2} 
	\end{equation}
	where the last inequality follows from (\ref{this}).

	Next, recall that  $$|\mu_1(\tilde{u}) - \mu(w_1)| \leq C\delta^{1/2}$$ for a solution $w_1$ of (\ref{ww1}). (See (\ref{EQ1}).) Similarly, one can prove  $$|\mu_1(u) - \mu(\tilde{w}_1)| \leq C N^{-1/2} \leq C \delta^{1/2}$$ where $\tilde{w}_1$ solves similar equations as in (\ref{ww1})  in the domain $\{-N\e \leq x \cdot \nu_1 \leq 0\}$, and the last inequality follows from (\ref{this}). Then since $w_1$ and $\tilde{w}_1$ have periodic Neumann data $G_1$ on  $H_0^{\nu_1}$, it corresponds to the case of $\nu =e_2$. Hence by similar arguments as in (\ref{sima}),
	$$|\mu(w_1) - \mu(\tilde{w}_1)|\leq C(\Lambda(\delta, e_2) + N^{-1}) \leq C(\delta^{1/2} + N^{-1}) \leq C \delta^{1/2} $$ and we can  conclude $$|\mu_1(u) -\mu_1(\tilde{u})| \leq  |\mu_1(u) -\mu(\tilde{w}_1)|+|\mu(\tilde{w}_1) -\mu(w_1)|+|\mu (w_1) -\mu_1(\tilde{u})| \leq C \delta^{1/2}.
	$$   
\end{proof}

\medskip

\noindent {\bf Step 2.  Constructing middle region barrier $\omega_\e$ (between $H_{-N \e /2}$ and $H_{-KmN\e }$)}

\medskip

In step 1 we showed that $N\e$ away from the boundary $H^{\nu_1}_0$, $u_\e^{\nu_1}$ is homogenized with average slope approximated by $\mu^N(G_k)$ in each vertical strip $I_k$.  Now more than $N\e$ away from $H^{\nu_1}_0$, we obtain the
second homogenization of $u_\e^{\nu_1}$, whose slope is determined by
$\mu^N (G_k)$, $k=1,..,m$.  Since  
the width of $I_k = N\e$,
the homogenized slopes $\mu^N (G_1)$,.., $\mu^N(G_m)$ are repeated $K$ times in a vertical strip of width $KmN\e$, $N\e$-away from  $H^{\nu_1}_0$.  We will specify 
$$
K:=1/\delta,
$$ but for computational clarity we will keep the symbol $K$.

\medskip

We will construct middle region barrier $\omega_\e$ in the region $\{-KmN\e \leq x_2 \leq -N\e /2\}$. To ensure that  $\omega_\e$ is regular near its Neumann boundary, we introduce a regularization of the original Neumann boundary data $\mu^N(G_k)$ as follows. 

\medskip

 Consider a ball $B_{\delta^{-\alpha_0 /2}N\e}(0)$. If  $I_k \cap H_0$, $I_j \cap H_0 \subset B_{\delta^{-\alpha_0 /2}N\e}(0)$, then  $|k-j| \leq \delta^{-\alpha_0/2}$ and 
\begin{equation}\label{ya}
|G_k(p, x/\e)-G_j(p,  x/\e) | \leq C(1+|p|) (|k-j|\delta)  \leq  C(1+|p|) \delta^{(1-\alpha_0/2)}. 
\end{equation}
Using this fact with Lemma~\ref{perturbation}, we can construct a $C^1$ function  $\Lambda(x)$ on $H_{-N\e/2}$ such that 

\begin{itemize} 
	\item[(a)] 
	$  \Lambda \in C^1(H_{-N\e/2})$ with  $\|\Lambda\|_{C^1}\leq \delta (N\e)^{-1}$;
	
	\item[(b)] $\mu^N(G_k)+\delta^{\alpha_0} \leq \Lambda (x) \leq \mu^N(G_k)+\delta^{\alpha_0}+\delta$ on each  $I_k$;
	
	\item[(c)] $\Lambda(x)$ is periodic with period $mN\e.$
\end{itemize}
Note that when we patch the middle region barrier  $\omega_\e$ with the near-boundary barrier $f_\e$ in step 6, we will need that the average slope of $\omega_\e$ is ``sufficiently'' larger than that of $f_\e$. For this, we will make the average slope of  $\omega_\e$ to be $\mu^N(G_k) + O(\delta^{\alpha_0})$, i.e.,  $(b)$ is to ensure that  $\mu_k(\omega_\e)$ is sufficiently larger than $\mu_k(f_\e)$. 
Also when we show the flatness of barriers in steps 4 and 5, we will localize them in a ``large'' ball of size $\delta^{-\alpha_0/2}N\e$. 

Let $\Sigma:=\{-K m N \e \leq x_2 \leq -N\e/2\}$ and $\omega_\e$ solve the following Neumann boundary problem 
\begin{equation}\label{middle}
\left\{\begin{array}{lll}
F(D^2\omega_\e,x/\e) =0 &\hbox{ in }& \Sigma \\ \\
\dfrac{\partial \omega_{\e}}{\partial x_2} = \Lambda(x)  &\hbox{ on }& H_{-N\e/2} \\ \\
\omega_{\e} =l(x) &\hbox{ on } & H_{-K m N \e}.
\end{array}\right.
\end{equation}

\medskip

\noindent {\bf Step 3. Homogenization of the operator in the middle region} 

\medskip

Next we show, similar to Lemma~\ref{important}, that the second homogenization does not change too much if the domain $\Pi$ is replaced by $\Pi^{\nu_1}$. More precisely, we will show that $\omega_\e$ is close to $\tilde{\omega}_\e$ solving

$$
\left\{\begin{array}{lll}
\bar{F}(D^2\tilde{\omega}_\e) =0 &\hbox{ in }& \{-KmN\e \leq x \cdot \nu_1 \leq  - N\e/2\}\\ \\
\dfrac{\partial \tilde{\omega}_\e}{\partial \nu_1} = \Lambda(x)  &\hbox{ on }& H^{\nu_1}_{-N\e/2} \\ \\
\tilde{\omega}_\e =l(x) &\hbox{ on } & H^{\nu_1}_{-KmN\e}. 
\end{array}\right.
$$
To this end we will first compare $\omega_\e$ with $\bar{\omega}_\e$, with the same Dirichlet data $l$  on $H_{-kmN\e}$ and solving

\begin{equation}
\label{middlebar}
\left\{\begin{array}{lll}
\bar{F}(D^2\bar{\omega}_\e) =0 &\hbox{ in }& \Sigma\\ \\
\dfrac{\partial \bar{\omega}_{\e}}{\partial x_2} = \Lambda(x)  &\hbox{ on }& H_{-N\e/2}.
\end{array}\right.
\end{equation}

\begin{lemma} \label{Lemma4.6}
	For any $\sigma>0$, there exists $N_0$ such that for $N_0>N$ we have
	$$
	|\omega_\e (x) -\bar{\omega}_\e(x)| \leq \sigma \delta N \e  
	\quad \hbox{ in } \quad
	\Sigma.$$
\end{lemma}
\begin{proof}
	The proof follows from Theorem~\ref{ext} applied to $(\delta N\e)^{-1} \omega_\e(N\e x)$. 
\end{proof}

Next we compare $\bar{\omega}_\e$ to  $\tilde{\omega}_\e$  to conclude. Here we will use the rotational invariance of $\bar{F}$.

\begin{lemma} \label{Lemma4.7}
	Let $\mathcal{O}$ be the rotation matrix that maps $e_2$ to $\nu_1$. Then
	$$ 
	| \tilde{\omega}_\e (\mathcal{O}x)- \bar{\omega}_\e(x)| \leq  \delta^{1/2}(KmN\e)
	$$
	in $\Sigma  \cap \{|x| \leq \delta^{-1/2} (N\e)\}$.
	
\end{lemma}

\begin{proof}
	Observe that $v(x):= \tilde{\omega}(\mathcal{O}x)$ solves $\bar{F}(v) = 0$ in $\Sigma$ with Neumann boundary data $\Lambda(\mathcal{O}x)$ on $H_{-N\e/2}$ and Dirichlet data $l(\mathcal{O}x)$ on $H_{-KmN\e}$. Note that due to \eqref{parameter} and the $C^1$ bound of $\Lambda$ we have 
	$$
	|\Lambda(KmN\e \mathcal{O}x) - \Lambda(KmN\e x) | \leq \theta_1 |KmN\e x| \sup |D\Lambda| \leq \delta |x|.
	$$
	and $|l(KmN\e \mathcal{O}x) - l(KmN\e x)| \leq KmN\e \theta_1 |x| \leq \delta |x|$.

	\medskip
	
	Hence one can apply Lemma 2.9 of \cite{CK} to $\tau^{-1} v(\tau x)$ and $\tau^{-1}\bar{w}(\tau x)$ in $\tau^{-1} \Sigma$, where $\tau = KmN\e$ and choose $R : = \delta^{-1/2}$ and $\e = 2$ to conclude.

\end{proof}

\medskip

\noindent {\bf Step 4.  Flatness of $\omega_\e$ on $H_{-N\e}$, and the construction of near-boundary barrier $f_\e$}

\medskip

\begin{lemma} \label{ffflat} [Flatness of $\omega_\e$]
	Let  $x_0$ be any point on $H_{-N \e}$. Then for
	$ x \in H_{-N\e} \cap B_{\delta^{-\alpha_0/2}N\e}(x_0)$
	$$
	|\omega_{\e}(x) - \omega_{\e}(x_0) -\partial_1\omega_\e(x_0)(x-x_0)_1| \leq
	C\delta^{1-\alpha_0}  N\e  .
	$$

\end{lemma}  

\begin{proof}

	Due to Lemma~\ref{Lemma4.6}, it is enough to show above lemma for $\bar{\omega}_\e$. 
	Let $\omega_1(x):= (KmN\e)^{-1}\bar{\omega}_\e((KmN\e) x)$, then   it solves
	$$
	\left\{\begin{array}{lll}
	\bar{F}(D^2 \omega_1) =0 &\hbox{ in }& \{-1 \leq x_2 \leq -\frac{1}{2Km}\}\\ \\
	\dfrac{\partial \omega_1}{\partial x_2} = \Lambda(KmN\e x)  &\hbox{ on }& H_{-\frac{1}{2Km}} \\ \\
	\omega_1(x)=l(x)+C &\hbox{ on } & H_{-1}
	\end{array}\right.
	$$
	we know that $\|\Lambda\|_{C^1}\leq \delta (N\e)^{-1}$, so the above Neumann boundary data has $C^1$ norm of $\delta Km$.
	From Theorem~\ref{thm:reg2}, we have that 
	$$
	\|\omega_1\|_{C^{1, 1}} \leq  C\delta Km .
	$$
	Hence
	\begin{equation}\label{flat0}
	|\omega_1(x)-\omega_1(x_0) - \partial_1 \omega_1(x_0)\cdot(x-x_0)| \leq C\delta Km |x-x_0|^{2} 
	\end{equation}
	which can be written in terms of $\bar{\omega}_\e$,

	$$
	\begin{array}{lll}
	|\bar{\omega}_\e(x) - \bar{\omega}_\e(x_0)  - \partial_1 \bar{\omega}_\e(x_0)\cdot (x-x_0)|  &\leq & C\delta (Km)^2 (N\e) |(KmN\e)^{-1} (x-x_0)|^{2}   \\ \\
	&\leq & C\delta  (\delta^{-\alpha_0/2})^{2} (N\e) = C\delta^{1-\alpha_0}N\e 
	\end{array}
	$$
	in $\delta^{-\alpha_0/2}N\e$-neighborhood of $x_0$. \end{proof}


Now we construct the near-boundary barrier $f_\e $ using $\rho_\e$.  Let $f_\e$ solve
$$
\left\{\begin{array}{lll}
F(D^2 f_\e, x/\e) =0 &\hbox{ in }& \{-N\e\leq x_2 \leq 0\};\\ \\
f_{\e}=\omega_{\e} + \delta^{1-\alpha_0}N\e   &\hbox{ on }& H_{-N\e };\\ \\
\dfrac{\partial f_{\e}}{\partial x_2} = G (Df_\e, \frac{x}{\e}  ) &\hbox{ on }& H_0.
\end{array}\right.
$$

\medskip

\noindent {\bf Step 5.  Flatness of $f_\e$}

\medskip

In this step we compare $\mu^N(G_k)$ given in \eqref{slope of v_k} with $\mu_k(f_\e)$ given in Definition~\ref{averaged_slope}. For simplicity we put $k=1$. 
Note that Lemma~\ref{lem:side},   Lemma~\ref{ffflat}, and Lemma~\ref{perturbation} with (\ref{ya}) 
imply that 
\begin{equation}\label{dede}
|\mu^N(G_1) -\mu_1(f_\e)| \leq C(\delta^{1-\alpha_0/2}  + \delta + \delta^{1-\alpha_0}) \leq C \delta^{1-\alpha_0}
\end{equation}
Also from Lemma~\ref{ffflat} and the definition of $f_\e$ it follows that $f_\e$ is close to a linear function

\begin{equation} \label{onH}
|f_\e (x)-L_0(x)| \leq C \delta^{1-\alpha_0}N\e   \,\,\, \hbox{ on
}\,\,\, H_{-N\e} \cap B_{ \delta^{-\alpha_0/2}N \e}(0),
\end{equation}
where $L_0(x): =f_\e(-N\e e_2)+ \mu^N(G_1)(x_2+N\e)+ \partial_1 f_\e(-N\e e_2)x_1.$ Then Lemma~\ref{Claim 22}, 
(\ref{onH}) and Lemma~\ref{lem:side}
applied to the rescaled function $(N\e)^{-1}f_{\e}(N\e x) $
in the region $\{-1\leq x_2 \leq -1/2\} \cap B_{ \delta^{-\alpha_0/2}}$ yield that
\begin{equation} \label{de}
|f_{\e}-L_0| \leq
C (\delta^{1-\alpha_0}  + \delta^{(1-\alpha_0/2)} ) N \e +C\e \leq C\delta^{1- \alpha_0} N\e
\end{equation}
in  $\{-N\e\leq x_2 \leq -N\e/2\} \cap B_{ \delta^{-\alpha_0/2}N \e}(0)$, where the last inequality follows from (\ref{this}).

\medskip

Before we proceed to the next step, observe that the $C^1$ regularity of $\Lambda$, Theorem~\ref{thm:reg2}, as well as Lemma~\ref{Lemma4.6} yield that

\begin{equation}\label{flat3}
|\omega_\e(x_1, x_2) - \omega_\e(x_1, -N\e) -\Lambda(x)(x_2+N\e)| \leq C\delta^{1-\alpha_0}N\e \quad \hbox{ on } \{-N\e \leq x_2 \leq -\frac{N\e}{2}\}.
\end{equation}

\medskip 

\noindent {\bf  Step 6. Patching up  } 

\medskip

Let $h(x):= l(x) + (\mu(\omega_\e) - C\delta^{1/2})(x_2+KmN\e)$ where $C>0$ is a  constant given as in (b) of Theorem~\ref{main-cell-2}, and $l(x)=l(x_1)$ is a linear function chosen so that $h(x)=q \cdot x$ on $H_{-1}$. 
We define 
\begin{equation*}
\rho_\e:=
\left\{\begin{array}{ll}
h &\hbox{ in }  \quad \{-1 \leq x_2 \leq -KmN\e\}, \\  \\ \omega_\e &\hbox{ in }  \quad \{-KmN\e \leq x_2 \leq -N\e /2\}.
\end{array}\right.
\end{equation*}
Since $\Lambda$ is $mN\e$-periodic,  (b) of Theorem~\ref{main-cell-2} implies that on  $\{x_2 = -KmN\e\}$,
$$
\partial_{x_2} \omega_ \e \geq \mu(\omega_\e) -
C \Lambda(1/K, e_2) = \mu(\omega_\e) -CK^{-1/2}  = \mu(\omega_\e) -C\delta^{1/2} =\partial_{x_2} h.$$ 
Thus it follows that $F(D^2 \rho_\e, \frac{x}{\e})\leq 0$ in $\{-1 \leq x_2 \leq -N\e/2\}$.   

\medskip

\medskip

Due to the flatness estimates \eqref{de} and \eqref{flat3}, we can approximate $f_\e$ and $\rho_\e$ by linear functions, respectively with normal derivatives of $\mu^N (G_k)$ and $\Lambda(x)$, with the error of $O(\delta^{1-\alpha_0} N\e)$. Here recall $\Lambda(x)$ was constructed so that $\Lambda(x) \geq \mu^N(G_k) + \delta^{\alpha_0}$, and $\alpha_0$ is a constant satisfying $ \alpha_0 \leq 1/2$. Then since $f_\e = \rho_\e + \delta^{1-\alpha_0}N\e$  on $\{x_2 = -N\e\}$,  
\begin{equation} \label{jj}
\rho_\e > f_\e \hbox{ on }  \{x_2 =-N\e /2\} \,\,\hbox{ and } \,\, f_\e > \rho_\e \hbox{ on } \{x_2 = -N\e\}.
\end{equation}
Define $\underline{\rho}_{\e} $ as follows:
$$
\underline{\rho}_{\e}:= \left\{\begin{array}{lll}
\rho_{\e}  &\hbox{ in } & \{-1 \leq x_2 \leq -N\e\},\\ \\
\min (\rho_\e, f_\e) &\hbox{ in } & \{ -N\e \leq x_2 \leq -N\e /2\},\\ \\
f_{\e} &\hbox{ in }& \{-N\e/2 \leq x_2 \leq 0\},
\end{array}\right.
$$
Then by (\ref{jj}),
$\underline{\rho}_\e$ is a viscosity supersolution of $(P)_{\e,e_2, 0,q}$ in $\{-1\leq x_2 \leq 0\}$. Let us mention that, due to Lemma~\ref{important}, Lemma~\ref{Lemma4.6} and Lemma~\ref{Lemma4.7}, a small perturbation of these barriers also yield a supersolution in $\{-1\leq x\cdot\nu_1 \leq 0\}$. 
Similarly, one can construct a subsolution $\bar{\rho}_{\e}$ of
$(P)_{\e,e_2, 0,q}$ by replacing $\Lambda(x)$ given in the construction of
$\rho_{\e}$ by $\tilde{\Lambda}(x) \leq  \mu^N(G_k) -\delta^{\alpha_0}$. Then by Lemma~\ref{important} and Lemma~\ref{Lemma4.7} 
\begin{equation}\label{comp1}
|\mu(u_\e^{\nu_1}) - \mu(\underline{\rho}_\e)| \leq |\mu(\bar{\rho}_\e) - \mu(\underline{\rho}_\e)| +C\delta^{1/2} \leq C(\delta^{1/2} + \delta^{\alpha_0}) \leq  C \delta^{\alpha_0} = C \delta^{1/2}
\end{equation}
where the last inequality follows by choosing $\alpha_0 =1/2$.

We denote $\bar{\rho}_{\e} = \bar{\rho}^{\nu_1}_{\e}$ and  $\underline{\rho}_{\e} = \underline{\rho}_{\e}^{\nu_1}$ indicating that they are obtained from the direction $\nu_1$, i.e., with the scale $N\e$.

\medskip

\noindent {\bf Step 7. Comparing the solutions $u_\e^{\nu_1}$ and  $u_\e^{\nu_2}$ : Proof  of Theorem~\ref{continuity} (a) }

\medskip 

Parallel arguments as in the previous steps apply to the other direction $\nu_2$. Recall that
$$
\theta_2 =|\nu_2-e_2| < \theta_1,  \quad M= [
\frac{\delta}{\theta_2}] > N.
$$
Then similarly as in the direction $\nu_1$, we can construct barriers $\bar{\rho}^{\nu_2}_{\e}$ and
$\underline{\rho}^{\nu_2}_{\e}$ such that

\begin{equation} \label{comp2}
|\mu(u_\e^{\nu_2}) - \mu(\underline{\rho}^{\nu_2}_\e)| \leq |\mu(\bar{\rho}^{\nu_2}_\e) - \mu(\underline{\rho}^{\nu_2}_\e)| +C\delta^{1/2} \leq C\delta^{1/2}.
\end{equation} Here  their corresponding Neumann
boundary conditions satisfy
$$
\mu^M(G_{k}) -\delta^{\alpha_0} -\delta \leq\,\,\,
\frac{\partial}{\partial x_2 } \bar{\rho}^{\nu_2}_{\e} \,\, \hbox{; }\,\,
\frac{\partial}{\partial x_2 } \underline{\rho}^{\nu_2}_{\e} \,\,\leq
\mu^M (G_{k}) +\delta^{\alpha_0} +\delta \quad \hbox{ on }
H_{-M\e}\cap I_k,
$$
where $\alpha_0 =1/2$, and the  respective derivatives of $\bar{\rho}^{\nu_2}_{\e}$ and $\underline{\rho}^{\nu_2}_{\e}$  are taken as a limit from the
region $\{-1\leq x_2 <-M\e\}$.

Thus to compare $\mu(u_\e^{\nu_1})$ and $\mu(u_\e^{\nu_2})$, we compare $\mu^N(G_k)$ and $\mu^M(
G_k)$. Recall that we define $\mu^M(G_{k})$ similarly as   $\mu^N(G_{k})$.   More precisely,  $\mu^M(G_{k})$ is the slope of the linear
approximation of $v_{k}^{M,\e}$, where $v_k^{M,\e}$ is defined similarly as in \eqref{vk} in the region $\{-M\e \leq x_2 \leq 0\}$ with the
boundary condition
$$\partial_{x_2} v_{k}^{M,\e}(x) =
G_k(D v_{k}^{M,\e}, x/\e) \quad \hbox{ on } \quad
H_0$$ and  $v_{k}^{M,\e} =l(x)$ on $H_{-M\e}$. 
Since $G_k$ is periodic on the Neumann boundary, it
corresponds to the case of Neumman boundary with rational normal,
passing through the origin. Hence by applying arguments as in the proof of (\ref{error222}), 
\begin{equation}
|\mu^N (G_{k}) - \mu^M(G_{k}) | \leq
C\Lambda(1/N, e_2) =C \inf_{0<k<1} \{1/N^k + 1/N^{1-k}\} =C /N^{1/2}  .\label{eqn5}
\end{equation}

Now we prove the following lemma using the estimate (\ref{eqn5}). 
\begin{lemma} For any $\e$ satisfying (\ref{order202}),
	$$|\mu(u^{\nu_1}_\e) -\mu(u^{\nu_2}_\e)| \leq C \delta^{1/2}.$$ 
\end{lemma}
\begin{proof}
	By the construction of the viscosity supersolution $\underline{\rho}^{\nu_1}_\e$ and Lemma~\ref{Lemma4.6}, 
	\begin{equation} \label{69}
	|\mu(\underline{\rho}^{\nu_1}_\e) -\mu(\bar{\omega}_\e)| \leq C \delta^{1/2}
	\end{equation}
	where $\bar{\omega}_\e$ is given as in (\ref{middlebar}). Similarly, we get
	\begin{equation}\label{70}
	|\mu(\underline{\rho}^{\nu_2}_\e) -\mu(\bar{\omega}^{\nu_2}_\e)| \leq C \delta^{1/2}
	\end{equation}
	where $\bar{\omega}^{\nu_2}_\e$  solves  
	\begin{equation*}
	\left\{\begin{array}{lll}
	\bar{F}(D^2\bar{\omega}^{\nu_2}_\e) =0 &\hbox{ in }& \{-K m M\e \leq x_2 \leq -M\e/2\};\\ \\
	\dfrac{\partial\bar{\omega}^{\nu_2}_{\e}}{\partial \nu} = \Lambda^{\nu_2}(x)  &\hbox{ on }& H_{-M\e/2} ;\\ \\
	\bar{\omega}^{\nu_2}_{\e} =l(x) &\hbox{ on } & H_{-K m M \e}.
	\end{array}\right.
	\end{equation*}
	Here $\Lambda^{\nu_2}(x) $ is constructed similarly as  $\Lambda(x)$ with $N$ replaced by $M$, i.e., with $\mu^N(G_k)$ replaced by $\mu^M(G_k)$. 
	Then  by \eqref{comp1},  \eqref{comp2}, \eqref{69}  and \eqref{70},  it suffices to prove 
	$$|\mu(\bar{\omega}_\e) - \mu(\bar{\omega}^{\nu_2}_\e)| \leq C \delta^{1/2}.$$
	Recall that $|\Lambda(x) -\mu^N(G_k)| \leq \delta^{\alpha_0} +\delta $ on $I_k$, and similarly,
	$|\Lambda^{\nu_2}(x) -\mu^M(G_k)| \leq \delta^{\alpha_0} +\delta $ on $I_k$, with $\alpha_0=1/2$. Hence 
	\begin{equation} \label{eqn7}
	|\mu(\bar{\omega}_\e)-\mu(h_1) |, \,\, |\mu(\bar{\omega}^{\nu_2}_\e)-\mu(h_2) | \leq C \delta^{1/2}
	\end{equation}
	for solutions $h_1$ and $h_2$ of
	$$\left\{\begin{array}{lll}
	\bar{F}(D^2 h_1) = 0 & \hbox { in } & \{ -KmN\e \leq x_2 \leq -N\e/2\}\\ \\
	\dfrac{\partial h_1}{\partial\nu } = \mu^N(G_{k}) &\hbox{
		on } & H_{-N\e/2}\cap I_k
	\\ \\
	h_1=l(x)  &\hbox{ on } & H_{-KmN\e}
	\end{array}\right.
	$$
	and
	$$
	\left\{\begin{array}{lll}
	\bar{F}(D^2 h_2) = 0 & \hbox { in } & \{ -KmM\e \leq x_2 \leq -M\e/2\}\\ \\
	\dfrac{\partial h_2}{\partial\nu }
	=\mu^M(G_{k})&\hbox{ on } & H_{-M\e/2}\cap I_k \\ \\
	h_2=l(x)  &\hbox{ on } & H_{-KmM\e}.
	\end{array}\right.
	$$
	Note that $h_1$ has a periodic Neumann condition on $H_{-N\e/2}$ with period $m
	N\e$, and also $h_2$ has a periodic Neumann condition on
	$H_{-M\e/2}$ with period $m M \e$. Hence they correspond to the case of periodic Neumann boundary data, i.e., the case of Neumann boundary with  a normal direction $e_2$, and  passing through the
	origin. Hence by Theorem~\ref{main-cell-2} with \eqref{eqn5} and $K =1/\delta$, we get
	\begin{equation} \label{eqn6}
	|\mu(h_1)-\mu(h_2)| \leq
	\Lambda(\delta, e_2)+C/N^{1/2} \leq
	C(\delta^{1/2}+(1/N)^{1/2}) \leq C \delta^{1/2}
	\end{equation}
	where the last inequality follows from (\ref{this}).
	Then we can conclude from (\ref{eqn7}) and (\ref{eqn6}).

\end{proof}

\medskip

\begin{remark} \label{generaln} For the dimension $n>2$ and $\nu = e_n$, for a fixed $m\in \NN$ and $\delta = \frac{1}{m}$ let us define
	$$
	G_i(x_1,...,x_{n-1}, x_n) := G(x_1,..., x_{n-1}, \delta(i-1))\quad \hbox{ for } i=0,...,m
	$$
	and
	$$
	I_{k_1, k_2, ..., k_{n-1}}:= [(k_1-1)N \e, k_1 N  \e]\times ...
	\times [(k_{n-1}-1) N \e, k_{n-1} N \e] \times \R.
	$$
	Then parallel arguments as in steps 1 to 9 would apply to yield the
	results in $\R^n$.
	
\end{remark}

\appendix

\section{Appendix} 

In this section we state quantitative results on distribution of $\e\Z^n$ near a hyperplane. We present an improved version from those introduced in \cite{CKL} . Recall that $\nu \in S ^{n-1}$ is a rational direction if $\nu \in \R \Z^n$, otherwise $\nu$ is an irrational direction.   For properties of irrational directions, let us discuss the averaging property of the sequence $(nx)_n$
mod $1$, for an irrational number $x$.   We are particularly interested in the
estimates on the rate of convergence of the sequence $(nx)_n$ to the
uniform distribution (Definition~\ref{discrepancy}). Note that the estimates in Lemma~\ref{lemma-M} below are  improved from the estimates in Lemma 2.7 of \cite{CKL}. We begin with
recalling the notion of equi-distribution.

\vspace{10pt}

\noindent $\bullet$ A bounded sequence $(x_1, x_2, x_3...)$ of
real numbers is said to be {\it equi-distributed} on an interval $[a,b]$
if for any $[c,d] \subset [a,b]$ we have
$$ \lim_{n \to \infty} \frac{|\{x_1,...,x_n\} \cap [c,d]|}{n}
=\frac{d-c}{b-a}.$$
Here $|\{x_1,...,x_n\} \cap [c,d]|$ denotes the number of elements of $\{x_1,...,x_n\} \cap [c,d]$.

\vspace{10pt}

\noindent $\bullet$ The sequence $(x_1, x_2, x_3,...)$ is said to be
equi-distributed modulo $1$ if $(x_1-[x_1], x_2-[x_2], ...)$ is
equi-distributed in the interval $[0,1]$.

\begin{lemma}[\cite{w}, Weyl's equidistribution theorem]\label{lem-weyl-1}
	If $a$ is an irrational number, the sequence  $(a, 2a, 3a, ...)$ is
	equi-distributed modulo $1$.
\end{lemma}

To discuss quantitative versions of Lemma~\ref{lem-weyl-1}, we introduce the notion of  discrepancy.  The
following definition is  from the book \cite{kn}.

\begin{definition}\label{discrepancy}
	Let $(x_k)_k$ be a sequence in $\RR$. For a subset $E
	\subset [0,1]$, let $$A(E;N) =|\{x_n: 1 \leq n \leq N\} \cap E|$$ i.e, $A(E;N)$  denotes the number of points $\{x_n\}$,
	$1\leq n \leq N$, that lie in $E$.
	\begin{itemize}
		\item[(a)]The sequence $(x_n)_n$ is said to be {\it uniformly distributed} modulo $1$ in $\R$ if
		$$ \lim_{N\to\infty} \dfrac{A(E;N)}{N} = \mu(E) $$
		for all $E = [a,b) \subset [0,1]$. Here $\mu$ denotes the Lebesgue measure.\\
		\item[(b)]For $x\in [0,1]$ and $N \in \NN$, the discrepancy $D_x(N)$ is defined as follows: 
		$$
		D_x(N) = \sup_{E=[a,b)} \Big|\dfrac{A(E;N)}{N}-\mu(E)\Big|,
		$$
		where $A(E;N)$ is defined with the sequence $(kx)_k$ modulo
		$1$.
	\end{itemize}
\end{definition}
It easily follows from Lemma~\ref{lem-weyl-1} that the sequence
$(x_k)_k=(kx)_k$ is uniformly distributed modulo $1$ for any
irrational number $x\in\R$. In particular $D_x(N)$ converges to zero
as $N\to\infty$.

\vspace{10pt}

Next, we apply the discrepancy function to multi-dimensions. For   a
direction $\nu =(\nu_1, ..., \nu_n) \in \mathcal{S}^{n-1}$, let
$\nu_i$ be the component with the biggest size, i.e.,
$$|\nu_i|=\max\{|\nu_j|: 1\leq j \leq n\}$$
(if there are multiple components then we choose $\nu_i$ with
the largest index $i$).
Let $H$ be the hyperplane in $\R^n$, which passes through $0$
and is normal to $\nu$, i.e.,
$$H =\{x \in \R^n : x\cdot \nu=0\}.$$
Since $\nu_i \neq 0$,  there exists $m=m(\nu)$ such that
$(1,...,1,m,1,...,1) \in H$, i.e.,
$$
(1,...,1,m,1,...,1)\cdot \nu=0
$$
where $m$ is the i-th
component of $(1,...,1,m,1,...,1)$. Note that $m$ is irrational iff $\nu$ is an irrational direction.  Define
\begin{equation}\label{mode}
\omega_{\nu}(N):=D_m(N).
\end{equation}
Note that if  $\nu$ is an irrational direction, then $$\omega_{\nu}(N)\to 0
\hbox{ as } N \to \infty.$$

\vspace{10pt} Now we are ready to state our quantitative estimate on
the averaging properties of the vector sequence $(n\nu)_n$ with an
irrational direction $\nu$.  Recall that for $\nu \in \mathcal{S}^{n-1}$ and $\tau \in \R^n$, 
$$\Pi(\nu,\tau)=
\{x: -1\leq (x-\tau)\cdot\nu \leq 0\}$$
and $$H_{0} =\{x: (x-\tau)\cdot\nu = 0\}.$$ Also  for $x_0 \in \Pi(\nu, \tau) $ and $d:= \rm{dist}(x_0, H_0)$, we  denote   
$$H_{d}= \{x: (x-\tau)\cdot \nu =d\}= \{x: (x-x_0)\cdot \nu =0\}.$$

\begin{lemma} \label{lemma-M}
	
	for $\nu \in \mathcal{S}^{n-1}$ and $\tau \in \R^n$, let $x_0 \in\Pi(\nu, \tau)$ and let   $0<\e< d:= {\rm dist}(x_0, H_{0})$.
	\begin{itemize}
		\item[(a)] Suppose that $\nu$ is a rational direction. Then for any
		$x \in H_{d}$, there is $y \in H_{d}$ such that
		$$|x-y| \leq T_\nu \e ;\,\,\, y = x_0 \hbox{ mod } \e \Z^n$$
		where $T_\nu$ is the smallest positive number  such that $T_{\nu}\nu \in \Z^n$.

		\item[(b)] Suppose that $\nu$ is an irrational direction, and let $\omega_\nu: \N \to\R^+$ be defined as in \eqref{mode}.
		Then there exists a dimensional constant $M >0$ such that the
		following is true: for any $x \in H_{d}$ and $N \in \NN$, there is $y \in \R^n$
		such that
		$$|x-y| \leq MN\e +\e \omega_\nu(N)
		;\,\,\, y = x_0 \hbox{ mod } \e \Z^n$$
		and
	$$
		{\rm dist}(y, H_{d}) <
		\e \omega_{\nu}(N).
	$$
		Here note that $\omega_\nu (N)$ converges to $0$ as $N \to \infty$.
		
		\item[(c)] If $\nu$ is an irrational direction, then for any $z \in \R^n$ and
		$\delta>0$, there is $w \in H_{d}$  such that
		$$|z-w| \leq \delta \hbox{ mod }\e\Z^n.
		$$
	\end{itemize}
\end{lemma}

\begin{proof}
	Proof of (a) is immediate from the fact that $T_{\nu}\nu \in \Z^n$.
		Next, we let $\nu$ be an irrational direction in
	$S^{n-1}$. Without loss of generality, we may assume
	$$|\nu_n|=\max\{|\nu_j|: 1\leq j \leq n\}.$$
	Let $x$ be any point on $H_{d}$: after a translation and rotation around the $x_n$-axis, we may assume
	that $x=0 \in H_{d}$ and $H_{d} \cap [0,1]^n \neq \emptyset$. Choose $m$ such that
	$$(1,1,..,1, m) \in H_{d}.$$ Note that $|(1,1,..,1,m)|\leq M$ for a dimensional constant $M>0$, since $|\nu_n|$ is the largest. Also note that
	$$k\e(1,1,..,1, m) \in H_{d}
	\hbox{ for any integer }k$$
	since $H_{d}$ contains the origin.
	
	Consider the sequence $(km)_k$, then from the definition of
	$\omega_\nu(N)$ and the discrepancy function $D_m(N)$, it follows
	that any interval $[a,b] \subset [0,1]$ of length $\omega_\nu(N)$
	contains at least one point $km $ (mod $1$), for some $k \leq
	N$.
	Hence for any
	$$z=(0,0,...,0,z_n) \hbox{ with }  0 \leq z_n \leq \e$$ there exists $$w=k\e(1,1,..,1, m)
	\in H_{d} $$  for some $0\leq k \leq N$ such that
	$$|z-w| \leq \e\omega_\nu(N) \,\, \hbox{  mod }\e\Z^n.$$
	
	Similarly, for any $z \in
	[0,\e]^n$, there exists $\tau \in H_{d} \cap[0,\e]^n $ such that $\tau = z + \alpha e_n$ with $|\alpha|\leq \e$. Then by the above argument, we can find  
	$$w = k\e(1,1,...,1,m) + \tau
	\in H_{d} $$ for some $0 \leq k \leq N$  such that
	\begin{equation}\label{zw}
	|z-w| \leq \e\omega_\nu(N)\hbox{ mod }\e\Z^n. 
	\end{equation}
	
	Now let $\tilde{x}_0$ be a point in $[0,\e]^n$ with $x_0 = \tilde{x}_0$ mod $\e\Z^n$, and we apply the above argument for $z = \tilde{x}_0$. Then we can find  $\tau \in H_{d} \cap[0,\e]^n $  and  
	$$w  = k\e(1,1,...,1,m) + \tau 
	\in H_{d} $$ with some $0 \leq k \leq N$  such that
	\begin{equation}\label{zw1}
	|x_0- w| = |\tilde{x}_0 -w| \leq \e\omega_\nu(N)\hbox{ mod }\e\Z^n. 
	\end{equation}

	On the other hand, recall that the coordinates are
	shifted so that $x=0$. Thus it suffices to find $y \in \R^n$ such
	that
	$$|x-y|=|y| \leq
	M N \e;\,\,\, y = x_0  \hbox{ mod }\e \Z^n$$ and $$ {\rm
		dist}(y, H_{d}) < \e\omega_{\nu}(N).$$
	By (\ref{zw1}), there exists $w \in H_{d}$ such that
	
	\begin{equation} \label{zzww}
	|x_0- w| \leq \e\omega_{\nu}(N) \,\, \hbox{ mod } \,\,
	\e\Z^n
	\end{equation}
	and
	\begin{equation}\label{zzw}
	|x- w| =|w| \leq MN \e
	\end{equation}
	where the last inequality follows from $|(1,1,..,1,m)| \leq M$ and $0 \leq k \leq N$.
	Given $w$ satisfying  (\ref{zzww}), we can take $y\in\R^n$ such that
	\begin{equation} \label{17}
	|x_0-y|=0\hbox{ mod }\e\Z^n,\hbox{ and }|y- w| \leq
	\e\omega_{\nu}(N).
	\end{equation}
	Then  by (\ref{zzw}) and (\ref{17})
	$$
	|x-y| = |y| \leq |y- w| +|w| \leq \e\omega_{\nu}(N) + MN \e.
	$$
	Also since  $w$ is contained in $  H_{d}$,
	$$
	{\rm dist}(y, H_{d}) \leq |y-w| \leq \e\omega_\nu(N).
	$$
	
	(c)  is a direct consequence of \eqref{zw} since $\omega_\nu(N) \to 0$ as $N \to \infty$.
\end{proof}

Next we state a version of Dirichlet's approximation theorem, whose proof is based on pigeon-hole principle. 
\begin{lemma} [Lemma 2.11 in  \cite{FK}] \label{FK}
	For $\alpha_1$,..,$\alpha_n \in \R$ and $N\in \N$, there are integers $p_1$,..., $p_n$, $q \in \Z$ with $1 \leq q \leq N$ such that   
	$$
	|q \alpha_i - p_i| \leq N^{-1/n} .
	$$
\end{lemma}

\end{document}